\documentclass[12pt]{amsart}

\usepackage{geometry}
\usepackage{amsmath}
\usepackage{amssymb}
\usepackage{amsthm}
\usepackage{amsfonts, dsfont}
\usepackage{paralist}
\usepackage{graphics} 
\usepackage{epsfig} 
\usepackage{graphicx}  
\usepackage{epstopdf}
\usepackage{epstopdf}
\usepackage{verbatim}
\usepackage{mathrsfs}
\usepackage{mathtools}
\usepackage{pstricks}
\usepackage{relsize}
\usepackage{tikz}
\usetikzlibrary{matrix}
\usepackage{subcaption}
\usepackage{pgfplots}
\usepackage{fixltx2e}
\usepackage{enumitem}
\usepackage{ upgreek }
\usepackage[colorlinks=true]{hyperref}
\hypersetup{urlcolor=blue, linkcolor=blue, citecolor=red}
\usepackage{bm}
\usepackage{appendix}

\usepackage{xcolor}

\newcommand{\D}[1]{\mbox{\rm #1}} 
\newcommand{\dd}{\D{d}}

\usepackage{soul}
\usepackage{color}

\usepackage[abbrev]{amsrefs}

\usepackage{amssymb}

\usepackage[all,cmtip]{xy}

\usepackage{tikz}
\usepackage{tikz-cd}

\numberwithin{equation}{section}

\theoremstyle{plain} 
\newtheorem{theorem}{Theorem}[section]
\newtheorem{lemma}[theorem]{Lemma}

\newtheorem{corollary}[theorem]{Corollary}

\newtheorem{remark}[theorem]{Remark}
\newtheorem{definition}[theorem]{Definition}

\definecolor{ForestGreen}{RGB}{34,139,34}
\definecolor{ao(english)}{rgb}{0.0, 0.5, 0.0}

\usepackage{lineno}

\begin{document}

\title[Non-Convex Global Optimization as Optimal Stabilization]{Non-Convex Global Optimization as an Optimal Stabilization Problem: Dynamical Properties}
\thanks{Y.H. has been supported by a Roth Scholarship from Imperial College London.}

\author{Yuyang Huang}
\address{Yuyang Huang \newline 
Department of Mathematics, Imperial College London, South Kensington Campus SW72AZ London, UK
}
\email{\texttt{yuyang.huang21@imperial.ac.uk}}
\thanks{ 
}

\author{Dante Kalise}
\address{Dante Kalise \newline 
Department of Mathematics, Imperial College London, South Kensington Campus SW72AZ London, UK
}
\email{\texttt{d.kalise-balza@imperial.ac.uk}}
\thanks{
}

\author{Hicham Kouhkouh}
\address{Hicham Kouhkouh \newline 
Department of Mathematics and Scientific Computing, NAWI, University of Graz, 
Heinrichstra{\ss}e 36, 8010, Graz, Austria
}
\email{\texttt{hicham.kouhkouh@uni-graz.at}}
\thanks{}







\date{\today}


\begin{abstract}
We study global optimization of non-convex functions through optimal control theory. Our main result establishes that (quasi-)optimal trajectories of a discounted control problem converge globally and practically asymptotically to the set of global minimizers. Specifically, for any tolerance $\eta > 0$, there exist parameters $\lambda$ (discount rate) and $t$ (time horizon) such that trajectories remain within an $\eta$-neighborhood of the global minimizers after some finite time $\tau$. This convergence is achieved directly, without solving ergodic Hamilton--Jacobi--Bellman equations. We prove parallel results for three problem formulations: evolutive discounted, stationary discounted, and evolutive non-discounted cases. The analysis relies on occupation measures to quantify the fraction of time trajectories spend away from the minimizer set, establishing both reachability and stability properties.
\end{abstract}



\maketitle

\section{Introduction}

Given a non-convex, continuous function $f:\mathbb{R}^{n}\to \mathbb{R}$, we are concerned with the solution of the global optimization problem
\begin{equation}\label{opt intro}
    \min\limits_{z\in\mathbb{R}^{n}}\, f(z)
\end{equation}
using optimal control methods. To do so, we introduce the following optimal stabilization problem:
\begin{equation}\label{OCP intro}
\begin{aligned}
    u_{\lambda}(x,t) := \inf_{\alpha(\cdot)} \; & \int_{0}^{t} \left(\frac{1}{2}|\alpha(s)|^{2} + f(y_{x}^{\alpha}(s))\right) e^{-\lambda s}\,\mathrm{d}s\,,\qquad \lambda>0\,, \\
    \text{subject to} \quad & \dot{y}_{x}^{\alpha}(s) = \alpha(s), \quad y_{x}^{\alpha}(0)=x\in\mathbb{R}^n\,,
\end{aligned}
\end{equation}
where $y_{x}^{\alpha}(\cdot)$ denotes the trajectory starting from $x$ with velocity $\dot{y}_{x}^{\alpha}(s) = \alpha(s)$, and $\alpha(\cdot):[0,\infty) \to \mathbb{R}^{n}$ is a measurable control. The value function $u_{\lambda}(x,t)$ represents the optimal cost-to-go over all such admissible trajectories.

Our goal is to investigate dynamical properties of (quasi-)optimal trajectories in \eqref{OCP intro} for global optimization. More precisely, our main results concern reachability (attainability) and stability of the set of global minimizers of the function $f(\cdot)$ that we wish to minimize, using (quasi-)optimal trajectories.  
We will consider three cases:
\begin{enumerate}[label = \underline{\textit{Case \arabic*}}]
    \item When $t<\infty$, and $\lambda>0$: this is the evolutive discounted problem whose value function shall be denoted by $u_{\lambda}(x,t)$. 
    \item\label{case 2} When $t=\infty$, and $\lambda>0$: this is the stationary discounted problem whose value function shall be denoted by $u_{\lambda}(x)$. 
    \item\label{case 3} When $t<\infty$, and $\lambda=0$: this is the evolutive non-discounted problem whose value function shall be denoted by $u(x,t)$. 
\end{enumerate}

At first sight, our method appears to replace a finite-dimensional optimization problem with one of infinite dimension. However, this reformulation is deliberate rather than heuristic. The setting we consider involves non-convex objective functions that may admit multiple global minimizers, making direct optimization particularly challenging. By embedding the problem into a family of optimal control systems, we \textit{lift} it to an infinite-dimensional framework where the dynamics can be analyzed. This control formulation naturally provides trajectories that provably converge to the set of global minimizers, offering a systematic way to navigate complex objective landscapes. It is worth to note that this approach is entirely deterministic, relying on control-theoretical principles rather than on stochastic or heuristic mechanisms. This viewpoint not only provides a rigorous foundation for the convergence analysis, but also offers a new dynamical interpretation of the optimization process.

In the companion manuscript \cite{huang2025convergence}, we focus on \ref{case 2}, and carry out a deeper analysis of the convergence properties, in particular by establishing explicit rates of convergence and refined quantitative estimates.

Throughout the manuscript we shall need some of the following assumptions on the function $f(\cdot)$ to be minimized, and we will refer to them wherever they are needed.

\textbf{Assumptions (A)}
\begin{enumerate}[label=\textbf{(A\arabic*)}]
\item\label{f: nice} $f : \mathbb{R}^n\to\mathbb{R}$ is  continuous and bounded that is
    \begin{equation*}
    \exists\;\underline{f},\,\overline{f}\; \text{ such that }\; \underline{f} \leq f(x) \leq \overline{f},\quad \forall\;x\in \mathbb{R}^n.
    \end{equation*}
\item\label{f: has min} $f$ attains the minimum, i.e.,  
  \begin{equation*}
    \mathfrak{M}:=\{x\in\mathbb{R}^{n}\,:\, f(x) = \underline{f}:= \min\limits_{z\in\mathbb{R}^{n}}f(z)\} \ne \emptyset .
\end{equation*}
\item\label{f: stab} For all $\delta>0$, there exists  $\gamma=\gamma(\delta)>0$ such that
    \begin{equation*}
         \inf\{f(x) - \underline{f}\,:\,\text{dist}(x,\mathfrak{M})\,> \delta\} \,>\, \gamma(\delta).
    \end{equation*}
\end{enumerate}
Assumption \ref{f: stab} is similar to  \cite[Assumption (H)]{bardi2023eikonal}. As it has been noted in \cite{bardi2023eikonal}: if $\mathfrak{M}$ is bounded, then  Assumption \ref{f: stab} is equivalent to $\liminf\limits_{|x|\to \infty} f(x)> \underline{f}$. The latter inequality is however not true when $\mathfrak{M}$ is unbounded. Assumption \ref{f: stab} expresses a form of finite energy gap or strict separation between the global minimizers of $f(\cdot)$ and all other regions of the state space. 
It ensures that outside any fixed neighborhood of the minimizer set $\mathfrak{M}$, the energy $f(\cdot)$ remains uniformly higher than its minimal value $\underline{f}$. More precisely, any state lying a fixed distance $\delta$ away from equilibrium (the minimizer set) requires at least an additional amount of energy $\gamma(\delta)$ compared to the minimal energy $\underline{f}$. In the context of Gibbs measures or small-noise asymptotics, this property prevents mass concentration away from $\mathfrak{M}$ and guarantees that, as temperature or noise vanishes, the measure concentrates exponentially near the set of global minimizers. 
In optimization terms, this condition expresses a strict separation between global and non-global states, quantifying how much worse any non-optimal configuration must be. It provides the necessary stability for the minimizers and for the validity of large-deviation or Laplace-type asymptotic analyses \cite[Theorem 3.9]{huang2025faithful}. 

\subsection{Contribution} The main result of this paper concerns global \textit{practical} asymptotic stability of the set $\mathfrak{M}$. This is a relaxed version of classical asymptotic stability, particularly when exact convergence to the equilibrium (the set $\mathfrak{M}$) is not expected or required (e.g., due to model uncertainty, disturbances, or computational constraints). The result reads as follows.

\begin{theorem}
Let Assumptions \textbf{(A)} be satisfied. 
Let $y_{x}^{\alpha}(\cdot)$ be an optimal trajectory for problem (\ref{OCP intro}). 
Then for all $\eta>0$, there exist $\lambda>0$ small, $t>0$ large, and $\tau = \tau(\eta, \lambda,t)>0$ such that 
\begin{equation*}
    \operatorname{dist}\left(y_{x}^{\alpha}(s), \mathfrak{M}\right) \leq \eta \quad \quad \forall\, s\in [\tau, t].
\end{equation*}
In other words, $y_{x}^{\alpha}(s) \in \left\{z\in \mathbb{R}^{n}: \operatorname{dist}(z,\mathfrak{M})\leq \eta\right\}$ for all $s\in [\tau, t]$. 

The same conclusion holds for \ref{case 2}, i.e. when $\lambda>0$ and $t=\infty$ in \eqref{OCP intro}, and for \ref{case 3}, i.e. when $\lambda = 0$ and $t<\infty$ in \eqref{OCP intro}.
\end{theorem}

We shall also provide a measure (in a suitable sense that we will make precise) for the fraction of time during which an optimal trajectory remains outside of a neighborhood of $\mathfrak{M}$, that is of the interval $[0,\tau]$. More precisely, our result shows that this measure vanishes when either $\lambda\to0$ or $t\to \infty$. 
Additionally, we establish our  results for the broader class of quasi-optimal trajectories, introduced in Definition \ref{def: quasi}. 

The main idea underlying the present manuscript is inspired by the approach developed in \cite{bardi2023eikonal}, where the authors design a gradient flow that drives any given initial state toward the set of global minimizers, denoted by $\mathfrak{M}$. Their method is based on analyzing the ergodic problem associated with a Hamilton--Jacobi--Bellman (HJB) equation. Specifically, they study two problems: one stationary, given by $u_{\lambda}(x)$, and one evolutionary, given by $u(x,t)$. By examining the asymptotic limits as $\lambda\to 0$ and $t\to\infty$, they derive a limiting function $v(x)$, which satisfies an ergodic HJB equation of Eikonal type. 
This limiting function is further shown to coincide with the value function of a time-optimal control problem with a free terminal time (i.e., a stopping time). It is this value function $v(x)$ that ultimately guides the system to $\mathfrak{M}$, thereby achieving global optimization through a rigorously justified dynamical mechanism.

In contrast, our approach demonstrates that the original problems, namely the stationary discounted problem $u_{\lambda}(x)$ and the time-dependent problem $u(x,t)$, in addition to the discounted evolutionary problem $u_{\lambda}(x,t)$, are each capable of directly steering the system to the global minimizers without the intermediary step of solving the ergodic problem. In particular, the function $u_{\lambda}(x,t)$ which combines temporal evolution with a discount parameter, plays a central role in unifying the analysis of the previous two cases. 

This distinctive and more direct route to global optimization is illustrated in Figure \ref{fig:placeholder}, where the dynamics induced by all three formulations are shown to asymptotically identify the set of global minimizers $\mathfrak{M}$ without passing through the ergodic limit, highlighting a more direct path to optimization within our framework.

\begin{figure}[h!]
    \centering
\includegraphics[width=0.5\linewidth]{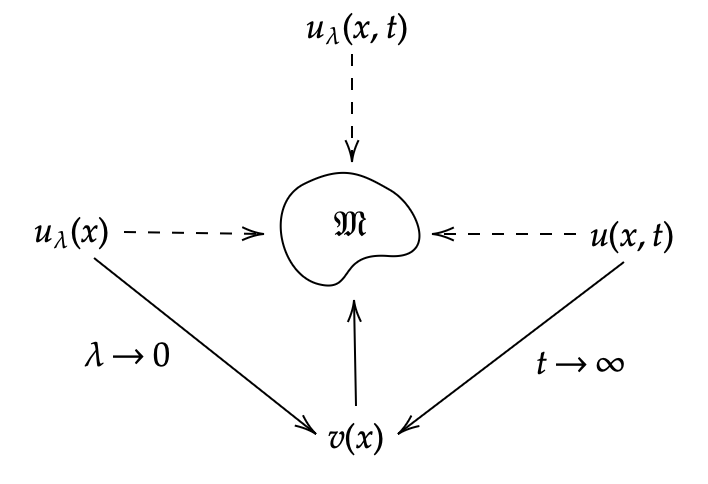}
    \caption{Illustration of the convergence relationships: \\The dashed arrows ``$\dashrightarrow$'' are those in the present manuscript.\\ The other arrows ``$\rightarrow$'' are those studied in \cite{bardi2023eikonal}.}
    \label{fig:placeholder}
\end{figure}

We note that the convergences $u_{\lambda}(x,t) \to u_{\lambda}(x)$ when $t\to \infty$, and $u_{\lambda}(x,t)\to u(x,t)$ when $\lambda\to 0$, can be established using standard arguments from HJB theory, thus completing the conceptual picture illustrated in Figure \ref{fig:placeholder}. However, we omit the details here, as these results are not central to our primary focus on optimization.

In the companion work \cite{huang2025convergence}, we focus on the stationary discounted case (\ref{case 2}), where we develop a deeper theoretical analysis of the convergence dynamics and derive explicit convergence rates as well as more precise quantitative results.

\subsection{Related literature}

Nonconvex optimization lies at the heart of modern data science, from neural-network training to inverse problems and complex engineering design. Despite substantial progress, computing global minimizers of highly nonconvex optimization problems remains challenging due to complex landscapes with multiple basins and extensive flat regions. In recent years, connections between nonconvex optimization and nonlinear partial differential equations (PDEs) have been systematically explored. 

For utilizing PDEs to solve nonconvex optimization problems, a well-established line of work links proximal operators and Moreau envelopes to viscous Hamilton–Jacobi (HJ) equations via the Hopf-Lax formula and the Cole-Hopf transformation \cite{chaudhari2018deep,osher2023hamilton,heaton2024global}. In this formulation, the Moreau envelope, which is an infimal convolution, arises as the viscosity solution of an HJ initial-value problem, the short-time viscous evolution produces a smooth surrogate of the objective while preserving its global structure. Laplace-type approximations of such infimal convolutions further motivate sampling-based optimization methods \cite{tibshirani2025laplace}. Building on these ideas, \cite{heaton2024global} show that viscosity widens narrow local valleys while preserving global minima, thereby facilitating escape from local traps. A complementary direction employs PDE convexification, where the convex envelope is characterized as the solution of a nonlinear (obstacle-type) PDE \cite{oberman2007convex,oberman2008computing,oberman2011dirichlet,oberman2017partial,carlier2012exponential}. 

Related developments have also emerged in machine learning. Entropy-SGD \cite{chaudhari2019entropy} admits a viscous Hamilton–Jacobi interpretation and   \textit{Deep Relaxation} formalizes this by casting training as a time-dependent viscous HJ PDE and analyzing the induced smoothing and generalization via tools from stochastic control and homogenization \cite{chaudhari2018deep}. Building on this perspective, control-theoretic refinements apply singular-perturbation arguments to controlled variants of SGD, yielding insights into stability and performance \cite{bardi2023singular,bardi2024deep}.

Diffusion-based smoothing and continuation provide another PDE pathway to mitigate nonconvexity. \cite{mobahi2015link} show that Gaussian homotopy continuation ( minimizing a sequence of Gaussian-smoothed objectives) acts as a heat-flow approximation to the convex-envelope evolution. Finally, PDEs also inform the design of accelerated dynamics. \textit{PDE Acceleration} derives momentum-like schemes from variational principles cast as PDEs \cite{benyamin2018accelerated,calder2019pde}. In a related spirit, collective-dynamics methods such as Consensus-Based Optimization (CBO) can be improved by optimal-feedback control derived from the HJB equations \cite{huang2024fast}.

The rest of the manuscript is organized as follows. In \textbf{Section \ref{sec: prelim}}, we collect several preliminary estimates, definitions, and auxiliary results that will be used throughout the paper. In particular, we recall the concept of occupation measures (Definition \ref{def: occ}) and prove the existence of optimal trajectories (Theorem \ref{thm: existence}). \textbf{Section \ref{sec: reach}} addresses the issue of \textit{reachability}, proving that trajectories can effectively reach the set of global minimizers $\mathfrak{M}$. Specifically, we show that the time spent far from $\mathfrak{M}$ becomes negligible as either $\lambda \to 0$ or $t \to 0$ (Theorems \ref{thm: stab lambda t}, \ref{thm: stab lambda}, and \ref{thm: stab t}). In \textbf{Section \ref{sec: stab}}, we establish a \textit{stability} property, showing that once trajectories enter a neighborhood of the minimizers, they remain within it for all subsequent times. Finally, \textbf{Section \ref{sec: conclusion}} combines the preceding results to prove the main \textit{convergence} theorem for quasi-optimal trajectories (Theorem \ref{thm: main conclusion}). For the sake of completeness, standard proofs of preliminary results are collected in \textbf{Appendix \ref{appendix}} and \textbf{Appendix \ref{appendix 2}}.

\section{Preliminaries}\label{sec: prelim}

Consider the optimal control problem \eqref{opt intro} with horizon $t<\infty$ and discount factor $\lambda>0$. Its value function $u_{\lambda}(x,t)$ is known to be the viscosity solution \cite{bardi1997optimal} of the Hamilton--Jacobi--Bellman (HJB) equation
\begin{equation}\label{eq: HJB lambda t}
\left\{\;
\begin{aligned}
    \partial_{t}u_{\lambda}(x,t) + \lambda\,u_{\lambda}(x,t) + \frac{1}{2}|Du_{\lambda}(x,t)|^{2} &= f(x),  &&\quad\quad (x,t) \in \mathbb{R}^{n}\times(0,+\infty)\\
    u_{\lambda}(x,0) &= 0,  &&\quad\quad x\in \mathbb{R}^{n},,
\end{aligned}
\right.
\end{equation}
where $D$ denotes the gradient with respect to the space variables $x$. 

When $t=\infty$ and $\lambda>0$ in \eqref{opt intro}, we shall denote the value function of the resulting optimal control problem by $u_{\lambda}(x)$. It is also the viscosity solution of the  HJB equation
\begin{equation}\label{eq: HJB lambda}
    \lambda\, u_{\lambda}(x) + \frac{1}{2}|Du_{\lambda}(x)|^{2} = f(x), \quad \quad x\in\mathbb{R}^{n}.
\end{equation}

When $t<\infty$ and $\lambda=0$ in \eqref{opt intro}, we shall denote the value function of the resulting optimal control problem by $u(x,t)$. It is also the viscosity solution of the HJB equation
\begin{equation}\label{eq: HJB t}
\left\{\;
\begin{aligned}
    \partial_{t}u(x,t)  + \frac{1}{2}|Du(x,t)|^{2} &= f(x),  &&\quad\quad (x,t) \in \mathbb{R}^{n}\times(0,+\infty)\\
    u(x,0) &= 0,  &&\quad\quad x\in \mathbb{R}^{n}.
\end{aligned}
\right.
\end{equation}

\subsection{Useful estimates}

Using the optimal control definition and the HJB characterization, we shall withdraw some useful bounds for the value functions. 

\begin{lemma}\label{lem: bounds}
Let assumption \ref{f: nice} be satisfied. 
Then the value functions $u_{\lambda}(x,t)$, $u_{\lambda}(x)$ and $u(x,t)$ satisfy the following bounds for almost every $(x,t)\in \mathbb{R}^{n}\times(0,+\infty)$
\begin{enumerate}[label=(\roman*)]
    \item Bound on the value function: 
    \begin{equation*}
        |\lambda\,u_{\lambda}(x,t)|,\, |\lambda\,u_{\lambda}(x)|\, \leq \|f\|_{\infty}
    \end{equation*}
    \item Bound on its time derivative: 
    \begin{equation*}
        |\partial_{t}u_{\lambda}(x,t)|,\, |\partial_{t}u(x,t)| \,\leq \|f\|_{\infty} 
    \end{equation*}
    \item Bound on its spatial derivative: 
    \begin{equation*}
        |D u_{\lambda}(x,t)|,  \,|D u_{\lambda}(x)|, \,|D u(x,t)| \,\leq \sqrt{6\|f\|_{\infty}}
    \end{equation*}
\end{enumerate}
\end{lemma}

The proof of this lemma relies on standard methods. For the sake of self-consistency of the paper and for the reader's convenience, we provide the proof in Appendix \ref{appendix}.

An important consequence of the latter result is the object of the next remark.

\begin{remark}\label{rem: bounded controls}
Given the bound $|D u_{\lambda}(x,t)| \leq \sqrt{6\|f\|_{\infty} }$, one can consider bounded controls in \eqref{opt intro} which thus can be equivalently expressed by
\begin{equation}\label{OCP M}
\begin{aligned}
    u_{\lambda}(x,t) = \; 
    \inf\limits_{\alpha(\cdot)} \; & \int_{0}^{t}\; \left(\frac{1}{2}|\alpha(s)|^{2} + f(y_{x}^{\alpha}(s))\right)\, e^{-\lambda s}\;\text{d}s,\\
    & \text{such that }\; \dot{y}_{x}^{\alpha}(s) = \alpha(s),\quad y_{x}^{\alpha}(0)=x\in\mathbb{R}^n\\
    & \text{and the controls } \alpha(\cdot):[0,\infty) \to B_{M} \text{ are measurable. } 
\end{aligned}
\end{equation}
where $B_{M}$ is the closed ball in $\mathbb{R}^{n}$ of radius $M \geq 1+\sqrt{6\|f\|_{\infty}}$. Indeed, the maximization in the Hamiltonian is achieved at an interior point, and one has
\begin{equation*}
\begin{aligned}
    H(x,r,p)
    & =\lambda r+\sup\limits_{\alpha\in B_{M}}\left\{ - \alpha \cdot p - \frac{1}{2}|\alpha|^{2}\right\}-f(x) = \lambda r+\sup\limits_{\alpha\in \mathbb{R}^{n}}\left\{ - \alpha \cdot p - \frac{1}{2}|\alpha|^{2}\right\}-f(x) \\
    & = \lambda r+\frac{1}{2}|p|^{2} - f(x)
\end{aligned}
\end{equation*}
Therefore we can assume without loss of generality that the controls are bounded by $M$, which is a constant that can be made arbitrarily large. 
\end{remark}

\subsection{Occupation measures}\label{sec: occ}

One of the main tools that we shall use to analyze the dynamical properties of the trajectories with respect to global optimization is the occupational measure. 

\begin{definition}\label{def: occ} 
Let $y_{x}^{\alpha}(\cdot)$ be an admissible trajectory starting from $x$ with the control $\alpha(\cdot)$. Let $\mathscr{O}$ be a Borel subset of $\mathbb{R}^{n}$, and  let $y\mapsto\mathds{1}_{\mathscr{O}}(y)$ be the indicator function which is equal to $0$ when $y\notin \mathscr{O}$, and is equal to $1$ otherwise.  We consider three cases for the problem \eqref{opt intro}.
\begin{enumerate}[label = \underline{Case \arabic*}.]
    \item  When $t<\infty$ and $\lambda>0$, we denote by $\Gamma_{t}^{\lambda}[x,\alpha](\cdot)$ the probability measure that is defined for any Borel subset $\mathscr{O}$ of $\mathbb{R}^{n}$ 
    \begin{equation*}
        \Gamma_{t}^{\lambda}[x,\alpha](\mathscr{O}) := \frac{\lambda}{1-e^{-\lambda t}} \int_0^t \mathds{1}_{\mathscr{O}}\left(y_{x}^{\alpha}(s)\right)\, e^{-\lambda s} \,\dd s.
    \end{equation*} 
    Equivalently, for any test function $\phi:\mathbb{R}^{n}\to \mathbb{R}$, for example in $L^{\infty}_{\text{loc}}(\mathbb{R}^{n})$, it holds
    \begin{equation*}
        \int_{\mathbb{R}^{n}} \phi(y) \, \dd \Gamma_{t}^{\lambda}[x,\alpha](y) = \frac{\lambda}{1-e^{-\lambda t}} \int_0^t \phi\left(y_{x}^{\alpha}(s)\right)\, e^{-\lambda s} \, \dd s.
    \end{equation*}
    \item When $t=\infty$ and $\lambda>0$, we denote by $\Gamma^{\lambda}[x,\alpha](\cdot)$  the probability measure that is defined for any Borel subset $\mathscr{O}$ of $\mathbb{R}^{n}$ 
    \begin{equation*}
        \Gamma^{\lambda}[x,\alpha](\mathscr{O}) := \lambda \int_0^{\infty} \mathds{1}_{\mathscr{O}}\left(y_{x}^{\alpha}(s)\right)\, e^{-\lambda s} \,\dd s.
    \end{equation*}
    Equivalently, for any test function $\phi:\mathbb{R}^{n}\to \mathbb{R}$, for example in $L^{\infty}_{\text{loc}}(\mathbb{R}^{n})$, it holds
    \begin{equation*}
        \int_{\mathbb{R}^{n}} \phi(y) \, \dd \Gamma^{\lambda}[x,\alpha](y) = \lambda \int_0^{\infty} \phi\left(y_{x}^{\alpha}(s)\right)\, e^{-\lambda s} \, \dd s.
    \end{equation*}
    \item When $t<\infty$ and $\lambda=0$, we denote by $\Gamma_{t}[x,\alpha](\cdot)$  the probability measure that is defined for any Borel subset $\mathscr{O}$ of $\mathbb{R}^{n}$ 
    \begin{equation*}
        \Gamma_{t}[x,\alpha](\mathscr{O}) := \frac{1}{t} \int_0^{t} \mathds{1}_{\mathscr{O}}\left(y_{x}^{\alpha}(s)\right)\,  \dd s.
    \end{equation*}
    Equivalently, for any test function $\phi:\mathbb{R}^{n}\to \mathbb{R}$, for example in $L^{\infty}_{\text{loc}}(\mathbb{R}^{n})$, it holds
    \begin{equation*}
        \int_{\mathbb{R}^{n}} \phi(y) \, \dd \Gamma_{t}[x,\alpha](y) =\frac{1}{t} \int_0^t \phi\left(y_{x}^{\alpha}(s)\right)\,  \dd s.
    \end{equation*}
\end{enumerate}
\end{definition}

The occupation measures quantify how much time (eventually in a discounted sense, when in the presence of $e^{-\lambda\,s}$) a given trajectory spends within some region of the state space. In particular, they represent the (discounted) fraction of time that the trajectory spends in some region (here $\mathscr{O}$).  The discount factor $e^{-\lambda\, s}$ reflects the idea that near future matters more than distant future: contributions occurring later are exponentially down-weighted, so earlier times count more. Consequently, the discounted time average emphasizes the beginning of the interval and fades as $s$ grows. We now give an equivalent probabilistic interpretation.
\begin{enumerate}[label = \textit{\underline{Case \arabic*}}.]
     \item  When $t<\infty$ and $\lambda>0$, it holds
    \begin{equation*}
    \Gamma_t^\lambda[x, \alpha](\mathscr{O})=\mathbb{P}\left(y_x^\alpha\left(\tau_t\right) \in \mathscr{O}\right) \quad \text { where } \tau_t \sim \operatorname{Exp}(\lambda) \text { conditioned on }\left\{\tau_t \leq t\right\} .
    \end{equation*}
    Here $\tau_t$ is a random variable that is exponentially distributed with parameter $\lambda$, over time interval $[0, t]$, equivalently, it has conditional density $q_{t, \lambda}(s)=\frac{\lambda e^{-\lambda s}}{1-e^{-\lambda t}}$ on $[0, t]$. Thus $\Gamma_t^\lambda[x, \alpha](\cdot)$ is the law of the trajectory evaluated at a discounted random time restricted to $[0, t]$. For any test function $\phi(\cdot)$, one also has
    \begin{equation*}
    \int_{\mathbb{R}^n} \phi(y) \, \dd \Gamma_t^\lambda[x, \alpha](y)=\mathbb{E}\left[\phi\left(y_x^\alpha\left(\tau_t\right)\right)\right] 
    \end{equation*}
    where $\tau_{t} \sim \operatorname{Exp}(\lambda)$ conditioned on $\left\{\tau_t \leq t\right\}$.
    \item When $t=\infty$ and $\lambda>0$, it holds
    \begin{equation*}
        \Gamma^{\lambda}[x,\alpha](\mathscr{O}) = \mathbb{P}\left(y_{x}^{\alpha}(\tau)\in \mathscr{O}\right) \quad \text{ where } \tau\sim \text{Exp}(\lambda).
    \end{equation*}
    Here $\tau$ is a random variable that is exponentially distributed with parameter $\lambda$. Thus $\Gamma^{\lambda}[x,\alpha](\mathscr{O})$ is the expected value of the trajectory at a random time governed by an exponential law. One can also write for a given test function $\phi(\cdot)$
    \begin{equation*}
        \int_{\mathbb{R}^{n}} \phi(y) \, \dd \Gamma^{\lambda}[x,\alpha](y) = \mathbb{E}\left[\phi(y_{x}^{\alpha}(\tau))\right]\quad \text{ where } \tau\sim \text{Exp}(\lambda)
    \end{equation*}
    which shows that $\Gamma^{\lambda}[x,\alpha](\cdot)$ is the law (distribution) of the trajectory when evaluated at an exponentially-distributed random time. 
    \item When $t<\infty$ and $\lambda=0$, it holds 
    \begin{equation*}
        \Gamma_{t}[x,\alpha](\mathscr{O}) = \frac{1}{t} \left|\{\,s\in [0,t]\,:\, y_{x}^{\alpha}(s) \in \mathscr{O}\}\right|
    \end{equation*}
    where $|\{\dots\}|$ denotes Lebesgue measure of an interval in $\mathbb{R}$. This means that $\Gamma_{t}[x,\alpha](\mathscr{O})$ counts how much time (in average) the trajectory spends in $\mathscr{O}$ during the period $[0,t]$. Additionally, the same probabilistic interpretation as in the previous case holds, the difference being that the random time is now uniformly distributed on the interval $[0,t]$. Hence we have
    \begin{equation*}
        \Gamma_{t}[x,\alpha](\mathscr{O}) = \mathbb{P}\left(y_{x}^{\alpha}(\tau)\in \mathscr{O}\right) \quad \text{ where } \tau\sim \text{Unif}([0,t]).
    \end{equation*}
    Here $\tau$ is a random variable that is uniformly distributed on $[0,t]$. Also for a given test function $\phi(\cdot)$, one has
    \begin{equation*}
        \int_{\mathbb{R}^{n}} \phi(y) \, \dd \Gamma_{t}[x,\alpha](y) = \mathbb{E}\left[\phi(y_{x}^{\alpha}(\tau))\right]\quad \text{ where } \tau\sim \text{Unif}([0,t])
    \end{equation*}
    which shows that $\Gamma_{t}[x,\alpha](\cdot)$ is the law (distribution) of the trajectory when evaluated at a uniformly-distributed random time. 
\end{enumerate}

The following lemma shows strong convergence of $\Gamma_{t}^{\lambda}[x,\alpha](\cdot)$, that is in \textit{Total Variation} (TV). Let us first recall the definition of the TV metric on the space of probability measures: given $\mu,\nu\in \mathscr{P}(\mathbb{R}^{n})$ two probability measures, the TV-metric $\|\cdot\|_{TV}$ is defined by
\begin{equation*}
    \|\mu - \nu\|_{TV} := \sup\limits_{\mathscr{O}}|\mu(\mathscr{O}) - \nu(\mathscr{O})|
\end{equation*}
where the supremum is taken over all Borel subsets $\mathscr{O}$ of $\mathbb{R}^{n}$. The induced convergence is 
\begin{equation*}
    \mu_{n} \xrightarrow[n\to \infty]{TV} \mu \quad \text{ when } \quad \|\mu_{n} - \mu\|_{TV} \to 0.
\end{equation*}

\begin{lemma}\label{lem: conv measures}
Let $x\in\mathbb{R}^{n}$ and $\alpha(\cdot)$ a control, both fixed. The following holds
\begin{equation*}
    \Gamma_{t}^{\lambda}[x,\alpha](\cdot) \;\xrightarrow[t\to \infty]{TV} \; \Gamma^{\lambda}[x,\alpha](\cdot) \quad \text{ and } \quad \Gamma_{t}^{\lambda}[x,\alpha](\cdot) \; \xrightarrow[\lambda\to 0]{TV} \; \Gamma_{t}[x,\alpha](\cdot).
\end{equation*}
\end{lemma}

The proof of this lemma can be found in Appendix \ref{appendix 2}.

In the particular situation where the occupational measure is evaluated on sets close to the global minimizers, we shall use the following simpler notations. Let $\delta\geq 0$ be fixed, and define the set of quasi-minimizers (or quasi-optimal sub-level set) 
\begin{equation}\label{eq: set K_delta}
    K_{\delta} := \{\,y\in\mathbb{R}^{n}\;|\; f(y) \leq \underline{f}\,+\delta\,\}
\end{equation}
where we recall $\underline{f} := \min\limits_{x\in\mathbb{R}^n}f(x)$. Clearly, when $\delta = 0$ one has $K_{0}=\mathfrak{M}$ as defined in \ref{f: has min}. Hence we denote the occupational measures evaluated in $K_{\delta}^{c}$, the complement of $K_{\delta}$, by
\begin{equation*}
    \mu_{t}^{\lambda}[x,\alpha](\delta) := \Gamma_{t}^{\lambda}[x,\alpha](K_{\delta}^{c}), \quad \mu^{\lambda}[x,\alpha](\delta) := \Gamma^{\lambda}[x,\alpha](K_{\delta}^{c}), \quad \mu_{t}[x,\alpha](\delta) := \Gamma_{t}[x,\alpha](K_{\delta}^{c}).
\end{equation*}

These probability measures record how often an admissible trajectory $\{s\mapsto y_{x}^{\alpha}(s) \, : \, \dot{y}_{x}^{\alpha}(\cdot) = \alpha(\cdot),\, y_{x}^{\alpha}(0)=x \}$ remains outside of a neighborhood $K_{\delta}$ of the set of global minimizers $\mathfrak{M}$. An application of Lemma \ref{lem: conv measures} yields
\begin{equation}\label{eq: conv mu}
    \mu_{t}^{\lambda}[x,\alpha](\delta) \xrightarrow[t \to \infty]{} \mu^{\lambda}[x,\alpha](\delta) \quad \text{ and } \quad \mu_{t}^{\lambda}[x,\alpha](\delta) \xrightarrow[\lambda \to 0]{} \mu_{t}[x,\alpha](\delta).
\end{equation}

\subsection{Trajectories}

We shall verify that the problem \eqref{OCP M} admits an optimal control-trajectory. Bearing in mind the conclusion of Remark \ref{rem: bounded controls}, we recall the set of admissible controls  
\begin{equation}\label{admissible control set}
    \mathcal{A} := \{\, \alpha(\cdot):[0,\infty)\to \mathbb{R}^{n}\,:\, \text{ measurable,} \, |\alpha(s)|\leq M \; \forall\,s \,\}.
\end{equation}
Consequently, the admissible trajectories are those whose speed (time derivative) is an element of $\mathcal{A}$
\begin{equation}\label{admissible state set}
    \mathcal{Y} := \{\, y(\cdot):[0,t]\to \mathbb{R}^{n}\,:\, \dot{y}(\cdot) \in \mathcal{A},\; y(0)\in \mathbb{R}^{n} \,\}.
\end{equation}
The running cost functional is defined on $\mathbb{R}^{n}\times\mathcal{A}$ such that
\begin{equation}\label{cost function}
    \mathscr{J}(x,\alpha(\cdot)) := \int_{0}^{t} \left( \frac{1}{2}|\alpha(s)|^{2} + f(y_{x}(s)) \right)\,e^{-\lambda\,s}\,\dd s \quad \text{ where } \dot{y}(\cdot) \equiv \alpha(\cdot), \quad y_{x}(0)=x\in\mathbb{R}^{n}
\end{equation}

\begin{theorem}\label{thm: existence}
Under Assumption \ref{f: nice}, there exists an optimal solution to \eqref{OCP M} in the three cases $(t<\infty, \lambda>0)$, $(t=\infty, \lambda>0)$, and $(t<\infty, \lambda=0)$.
\end{theorem}

\begin{proof}
We write the proof when $(t<\infty, \lambda>0)$. It follows the standard strategy in optimal control or calculus of variations: Starting from a minimizing sequence, we invoke weak-$*$ compactness to extract a converging subsequence, whose limit is in the admissible set when the latter is closed. Then we verify lower semicontinuity of the cost functional, before passing to the limit. Throughout the proof, the initial position $x\in\mathbb{R}^{n}$ will remain fixed, and thus we will omit the subscript $x$ when denoting the trajectories. 

Let us consider a minimizing sequence $\{\alpha_{k}(\cdot)\}\subset \mathcal{A}$ that is
\begin{equation*}
    \lim\limits_{k\to \infty} \mathscr{J}(x,\alpha_{k}(\cdot)) = \inf\limits_{\alpha(\cdot)} \mathscr{J}(x,\alpha(\cdot)) = u_{\lambda}(x,t).
\end{equation*}
Since $\|\alpha_{k}\|_{\infty}\leq M$, the admissible set of controls $\mathcal{A}$ is compact in the weak-$*$ topology. thus we can use Banach-Alaoglu theorem to extract a converging subsequence (again denoted by $\alpha_{k}$) such that $\alpha_{k} \xrightharpoonup{*} \alpha_{*}$, and $\alpha_{*}\in \mathcal{A}$ since the latter is closed under weak-$*$ convergence. 

The corresponding sequence of trajectories $\{y_{k}(\cdot)\}\subset \mathcal{Y}$ is $\;y_{k}(s) := x + \int_{0}^{s} \alpha_{k}(r)\,\dd r.\;$ 
These are equicontinuous and uniformly bounded. Hence Ascoli-Arzel\`a theorem ensures (eventually after extracting a subsequence) that $y_{k}\to y_{*}$ uniformly on $[0,t]$, and $\;y_{*}(s) = x + \int_{0}^{s} \alpha_{*}(r)\,\dd r.$

Now we check the lower semicontinuity. Recalling $\alpha_{k} \xrightharpoonup{*} \alpha_{*}$ in $L^{\infty}$, and since $L^{\infty}\subset L^{2}$, we also have the weak convergence $\alpha_{k} \rightharpoonup \alpha_{*}$ in $L^{2}$. Observe that the map $\alpha(\cdot) \mapsto \int_{0}^{t} |\alpha(s)|^{2}\, e^{-\lambda\,s} \,\dd s$ is a convex integral functional, hence  weakly lower semicontinuous in $L^{2}$ i.e.
\begin{equation*}
    \int_{0}^{t} |\alpha_{*}(s)|^{2}\, e^{-\lambda\,s}\,\dd s \leq \liminf\limits_{k\to \infty} \int_{0}^{t} |\alpha_{k}(s)|^{2}\, e^{-\lambda\,s}\,\dd s.
\end{equation*}
On the other hand, $f(\cdot)$ being continuous and $y_{k}\to y_{*}$ uniformly imply $f(y_{k}(s)) \to f(y_{*}(s))$ pointwise. Moreover, $f(\cdot)$ being bounded, we can apply dominated convergence theorem and obtain
\begin{equation*}
    \int_{0}^{t} f(y_{k}(s))\,e^{-\lambda\, s}\,\dd s \to \int_{0}^{t} f(y_{*}(s))\,e^{-\lambda\, s}\,\dd s.
\end{equation*}
Finally, it holds
\begin{equation*}
    \mathscr{J}(x,\alpha_{*}) = \int_{0}^{t} \left(\frac{1}{2}|\alpha_{*}(s)|^{2} + f(y_{*}(s))\right)\, e^{-\lambda\,s}\,\dd s \leq \liminf\limits_{k\to \infty} \int_{0}^{t} \left(\frac{1}{2}|\alpha_{k}(s)|^{2} + f(y_{k}(s))\right)\, e^{-\lambda\,s}\,\dd s 
\end{equation*}
that is $\mathscr{J}(x,\alpha_{*}) \leq u_{\lambda}(x,t)$, and $(\alpha_{*}(\cdot), y_{*}(\cdot))$ is an optimal control-trajectory.

When $(t=\infty, \lambda>0)$, or $(t<\infty, \lambda=0)$, the same proof remains valid. The only difference is in the case where $t=\infty$: 
the convergence $y_{k}\to y_{*}$ should be locally-uniformly on $[0,\infty)$.
\end{proof}

While the existence of optimal trajectories can be guaranteed under standard conditions as in the previous theorem, their practical computation is often hindered by the limitations of numerical approximation methods 
and potential sensitivity to modeling errors. To address this, we also consider quasi-optimal trajectories, which provide a tractable alternative and are defined as follows.

\begin{definition}\label{def: quasi}
Let $\varepsilon>0$ be fixed and $\alpha_{\varepsilon}(\cdot)$ be an admissible control. The corresponding trajectory $y_{x}^{\alpha_{\varepsilon}}(\cdot)$ is called quasi-optimal (or $\varepsilon$-optimal) if it satisfies the following.
\begin{enumerate}[label = \underline{Case \arabic*}.]
    \item When $t<\infty$ and $\lambda>0$ in \eqref{opt intro}, it holds
    \begin{equation*}
        \int_{0}^{t}\; \left(\frac{1}{2}|\alpha_{\varepsilon}(s)|^{2} + f(y_{x}^{\alpha_{\varepsilon}}(s))\right)\, e^{-\lambda s}\;\text{d}s \leq u_{\lambda}(x,t)+ \varepsilon.
    \end{equation*}
    \item When $t=\infty$ and $\lambda>0$ in \eqref{opt intro}, it holds
    \begin{equation*}
        \int_{0}^{\infty}\; \left(\frac{1}{2}|\alpha_{\varepsilon}(s)|^{2} + f(y_{x}^{\alpha_{\varepsilon}}(s))\right)\, e^{-\lambda s}\;\text{d}s \leq u_{\lambda}(x)+ \varepsilon.
    \end{equation*}
    \item When $t<\infty$ and $\lambda=0$ in \eqref{opt intro}, it holds
    \begin{equation*}
        \int_{0}^{t}\; \frac{1}{2}|\alpha_{\varepsilon}(s)|^{2} + f(y_{x}^{\alpha_{\varepsilon}}(s))\,\text{d}s \leq u(x,t)+ \varepsilon.
    \end{equation*}
\end{enumerate}
\end{definition}

Clearly, when the above inequalities are satisfied with $\varepsilon=0$, the trajectory is then optimal. 

We are now ready to state and prove our main results in the subsequent sections. 

\section{Reachability of the set of global minimizers}\label{sec: reach}

\subsection{The finite-horizon discounted problem}

Recalling the control problem \eqref{opt intro}, and bearing in mind the conclusion of Remark \ref{rem: bounded controls}, we shall consider in this subsection the case where $t<\infty$ and $\lambda>0$. The optimal control problem is then
\begin{equation}\label{OCP time lambda M}
\begin{aligned}
    u_{\lambda}(x,t) = \; 
    \inf\limits_{\alpha(\cdot)} \; & \int_{0}^{t}\; \left(\frac{1}{2}|\alpha(s)|^{2} + f(y_{x}^{\alpha}(s))\right)\, e^{-\lambda s}\;\text{d}s,\\
    & \text{such that }\; \dot{y}_{x}^{\alpha}(s) = \alpha(s),\quad y_{x}^{\alpha}(0)=x\in\mathbb{R}^n\\
    & \text{and the controls } \alpha(\cdot):[0,\infty) \to B_{M} \text{ are measurable. } 
\end{aligned}
\end{equation}

Let us be given $\varepsilon>0$ fixed, and consider an $\varepsilon$-optimal trajectory $y_{x}^{\alpha_{\varepsilon}}(\cdot)$ for the problem \eqref{OCP time lambda M}, as in \textit{Case 1} of Definition \ref{def: quasi}. Recall the occupation measure
\begin{equation*}
    \mu_{t}^{\lambda}[x,\alpha_{\varepsilon}](\delta) = \Gamma_{t}^{\lambda}[x,\alpha_{\varepsilon}](K_{\delta}^{c}) \; \in [0,1]
\end{equation*}
where $K_{\delta}$ is the set of quasi-minimizers defined in \eqref{eq: set K_delta}, and $\Gamma_{t}^{\lambda}[x,\alpha](\cdot)$ is in \textit{Case 1} of Definition \ref{def: occ}. 

The following theorem provides a bound on the fraction of time that an $\varepsilon$-optimal trajectory spends away from the set of global minimizers.

\begin{theorem}\label{thm: stab lambda t}
Let Assumptions \ref{f: nice} and \ref{f: has min} hold. 
Let $\varepsilon>0$ be given, and $y_{x}^{\alpha_{\varepsilon}}(\cdot)$ be an $\varepsilon$-optimal trajectory for \eqref{OCP time lambda M} starting from $x$. Define $R:= \sqrt{6\|f\|_{\infty}}$. Then the corresponding occupational measure satisfies
\begin{equation*}
    \mu_{t}^{\lambda}[x,\alpha_{\varepsilon}](\delta) \leq \frac{\lambda}{1-e^{-\lambda t}}\left(\frac{R \operatorname{dist}(x,\mathfrak{M}) + \varepsilon}{\delta}\right) \quad\quad \forall\; \delta>0.
\end{equation*}
In particular, when $\varepsilon=0$, i.e. for an optimal trajectory $y_{x}^{\alpha_{*}}(\cdot)$, one obtains 
\begin{equation*}
    \mu_{t}^{\lambda}[x,\alpha_{*}](\delta) \leq \frac{\lambda}{1-e^{-\lambda t}}\,\frac{R \operatorname{dist}(x,\mathfrak{M})}{\delta} \quad\quad \forall\; \delta>0.
\end{equation*}
\end{theorem}

\begin{proof}
Let $y_{x}^{\alpha_{\varepsilon}}(\cdot)$ be an $\varepsilon$-optimal trajectory for the problem \eqref{OCP time lambda M}. Thus it holds
\begin{equation*}
    \int_{0}^{t}\; f(y_{x}^{\alpha_{\varepsilon}}(s))\, e^{-\lambda s}\;\dd s  
     \leq \int_{0}^{t}\; \left(\frac{1}{2}|\alpha_{\varepsilon}(s)|^{2} + f(y_{x}^{\alpha_{\varepsilon}}(s))\right)\, e^{-\lambda s}\;\dd s 
    \leq u_{\lambda}(x,t)+ \varepsilon
\end{equation*}
Recalling $\underline{f} = \min_{z\in\mathbb{R}^{n}} f(z)$, we subtract $\underline{f}$ inside the integral and obtain 
\begin{equation}\label{eq: bound in proof of reach}
\begin{aligned}
    \int_{0}^{t}\; \left(f(y_{x}^{\alpha_{\varepsilon}}(s)) - \underline{f}\right)\, e^{-\lambda s}\;\dd s  
    & \leq u_{\lambda}(x,t)-\frac{1-e^{-\lambda t}}{\lambda}\,\underline{f} + \varepsilon.
\end{aligned}
\end{equation}
Next we split the integral in the left hand-side 
\begin{equation*}
\begin{aligned}
    & \int_{0}^{t}\; \left(f(y_{x}^{\alpha_{\varepsilon}}(s)) - \underline{f}\right)\, e^{-\lambda s}\;\dd s\\
    &\quad  = \int_{0}^{t}\; \mathds{1}_{K_{\delta}}(y_{x}^{\alpha_{\varepsilon}}(s))\left(f(y_{x}^{\alpha_{\varepsilon}}(s)) - \underline{f}\right) e^{-\lambda s}\dd s + \int_{0}^{t}\; \mathds{1}_{K_{\delta}^{c}}(y_{x}^{\alpha_{\varepsilon}}(s))\left(f(y_{x}^{\alpha_{\varepsilon}}(s)) - \underline{f}\right) e^{-\lambda s}\dd s.
\end{aligned}
\end{equation*}
The first integral is always non-negative by definition of $\underline{f}$. For the second integral, when $y_{x}^{\alpha_{\varepsilon}}(s)\in K_{\delta}^{c}$, we have $f(y_{x}^{\alpha_{\varepsilon}}(s)) - \underline{f} > \delta$. Therefore we have
\begin{equation*}
    \int_{0}^{t}\; \left(f(y_{x}^{\alpha_{\varepsilon}}(s)) - \underline{f}\right)\, e^{-\lambda s}\;\dd s \geq \delta\; \int_{0}^{t}\; \mathds{1}_{K_{\delta}^{c}}(y_{x}^{\alpha_{\varepsilon}}(s))\; e^{-\lambda s}\;\dd s = \delta\; \frac{1-e^{-\lambda t}}{\lambda} \;\mu_{t}^{\lambda}[x,\alpha_{\varepsilon}](\delta).
\end{equation*}
Using the latter lower-bound in \eqref{eq: bound in proof of reach} yields
\begin{equation}\label{eq: almost there}
    \delta\; \frac{1-e^{-\lambda t}}{\lambda} \;\mu_{t}^{\lambda}[x,\alpha_{\varepsilon}](\delta) \leq \left(u_{\lambda}(x,t)-\frac{1-e^{-\lambda t}}{\lambda}\,\underline{f}\right) + \varepsilon.
\end{equation}
Let us now consider an arbitrary element $z\in \mathfrak{M}$ whose existence is guaranteed thanks to assumption \ref{f: has min}. It can be easily verified that $u_{\lambda}(z,t) = \frac{1-e^{-\lambda t}}{\lambda}\,\underline{f}\,$ which corresponds to an optimal control $\alpha_{*}(\cdot)\equiv 0$. Therefore we have
\begin{equation*}
    u_{\lambda}(x,t)-\frac{1-e^{-\lambda t}}{\lambda}\,\underline{f} = u_{\lambda}(x,t)- u_{\lambda}(z,t) \leq R\, |x-z|
\end{equation*}
where $R:= \sqrt{6\|f\|_{\infty}}$ is the bound in \textit{(iii)} of Lemma \ref{lem: bounds}. On the other hand, recalling $\underline{f} \leq f(x)$, one gets for any $\alpha(\cdot)$ 
\begin{equation}
    \frac{1-e^{-\lambda t}}{\lambda} \; \underline{f}\; \leq  \int_{0}^{t}\; \left(\frac{1}{2}|\alpha(s)|^{2} + f(y_{x}^{\alpha}(s))\right)\, e^{-\lambda s}\;\text{d}s.
\end{equation}
Taking the infimum over $\alpha(\cdot)$ yields $\frac{1-e^{-\lambda t}}{\lambda}\,\underline{f} \leq u_{\lambda}(x,t)$. 
Hence we have
\begin{equation*}
    \left|u_{\lambda}(x,t)-\frac{1-e^{-\lambda t}}{\lambda}\,\underline{f}\right| \leq R\, |x-z|
\end{equation*}
The latter holds regardless of the point $z\in \mathfrak{M}$. Consequently we have
\begin{equation*}
    \left|u_{\lambda}(x,t)-\frac{1-e^{-\lambda t}}{\lambda}\,\underline{f}\right| \leq R\, \operatorname{dist}(x,\mathfrak{M}).
\end{equation*}
Going back to \eqref{eq: almost there}, one finally gets
\begin{equation*}
    \mu_{t}^{\lambda}[x,\alpha_{\varepsilon}](\delta) \leq \frac{R}{\delta}\, \operatorname{dist}(x,\mathfrak{M}) \, \frac{\lambda}{1-e^{-\lambda t}} + \frac{\varepsilon}{\delta}\,\frac{\lambda}{1-e^{-\lambda t}}.
\end{equation*}
\end{proof}

\begin{corollary}\label{cor: stab lambda t}
In the situation of Theorem \ref{thm: stab lambda t}, if additionally assumption \ref{f: stab} holds, then for all $\eta>0$, there exist a large $t>0$, a small $\lambda>0$, and some $s\in [0,t]$ such that
\begin{equation*}
    \operatorname{dist}\left(y_{x}^{\alpha_{\varepsilon}}(s),\mathfrak{M}\right)\leq \eta.
\end{equation*}
The same result is satisfied by an optimal trajectory. 
\end{corollary}

This result shows that one can always reach a neighborhood of the set of global minimizers with a margin $\eta>0$ that can be made arbitrarily small, provided the parameters $t$ and $\lambda$ are well  tuned. More importantly, this reachability is possible even with quasi-optimal trajectories, hence there is no need in practice for computing the exact optimal solutions of the control problem \eqref{OCP time lambda M}. In other words, quasi-optimality in the control problem \eqref{OCP time lambda M} allows for optimality in the optimization problem $\min_{z\in\mathbb{R}^{n}}f(z)$. 

\begin{remark}
An example of how small $\lambda$ or how large $t$ could be chosen is the following
\begin{equation}
    \label{config: lambda t}
    0<t = \frac{1}{\lambda} \quad \text{ and } \quad 0<\lambda \leq \frac{1-e^{-1}}{2}\left(\frac{R \operatorname{dist}(x,\mathfrak{M}) + \varepsilon}{\gamma(\eta)}\right)^{-1}.
\end{equation}
This example serves to illustrate a possible scenario rather than to prescribe choices of $\lambda$ and $t$. 
This also highlights the difference between choosing $\varepsilon$-optimal trajectories or optimal ones: in the former case (with $\varepsilon>0$) $\lambda$ needs to be smaller than in the latter case (where $\varepsilon=0$).
\end{remark}

\begin{proof}[Proof of Corollary \ref{cor: stab lambda t}] 
We proceed by contradiction. Suppose that there exists $\eta>0$ such that for all $t>0$ and $\lambda>0$, one has $\operatorname{dist}\left(y_{x}^{\alpha_{\varepsilon}}(s),\mathfrak{M}\right)> \eta$ for all $s\in [0,t]$. Then using assumption \ref{f: stab}, there exists $\gamma(\cdot)>0$ such that
\begin{equation}\label{eq: contradiction gamma}
     f(y_{x}^{\alpha_{\varepsilon}}(s)) - \underline{f} \geq 
     \inf\{f(x) - \underline{f}\,:\,\text{dist}(x,\mathfrak{M})\,> \eta\} \,>\, \gamma(\eta) \quad \quad \forall\,s\in [0,t]
\end{equation}
Consider the set of quasi-minimizers (see \eqref{eq: set K_delta}) corresponding to $\gamma(\eta)$
\begin{equation*}
    K_{\gamma(\eta)} = \{y\in \mathbb{R}^{n}\,:\, f(y) \leq \underline{f}+\gamma(\delta)\} \quad \text{ and its complement } K_{\gamma(\delta)}^{c} = \mathbb{R}^{n}\setminus K_{\gamma(\delta)}. 
\end{equation*}
Then from \eqref{eq: contradiction gamma} it holds $\mathds{1}_{K_{\gamma(\eta)}^{c}}(y_{x}^{\alpha_{\varepsilon}}(s)) = 1$ for all $s\in [0,t]$. Consequently the occupational measure evaluated in $\gamma(\eta)$ satisfies $\mu_{t}^{\lambda}[x,\alpha_{\varepsilon}](\gamma(\eta)) = 1$. Recalling the result of Theorem \ref{thm: stab lambda t}, one gets
\begin{equation*}
    1 \leq \frac{\lambda}{1-e^{-\lambda t}}\left(\frac{R \operatorname{dist}(x,\mathfrak{M}) + \varepsilon}{\gamma(\eta)}\right).
\end{equation*}
However when $\lambda$ is small and $t$ is large, one gets a contradiction. An example of such configuration is the one in \eqref{config: lambda t}, which then yields $\frac{\lambda}{1-e^{-\lambda t}}\left(\frac{R \operatorname{dist}(x,\mathfrak{M}) + \varepsilon}{\gamma(\eta)}\right) \leq \frac{1}{2}$.
\end{proof}

\subsection{The infinite-horizon discounted problem}

Similarly to the previous subsection, we shall now consider the case where $t=\infty$ and $\lambda>0$ in \eqref{opt intro}. The resulting optimal control problem is then
\begin{equation}\label{OCP lambda}
\begin{aligned}
    u_{\lambda}(x) = \; 
    \inf\limits_{\alpha(\cdot)} \; & \int_{0}^{\infty}\; \left(\frac{1}{2}|\alpha(s)|^{2} + f(y_{x}^{\alpha}(s))\right)\, e^{-\lambda s}\;\text{d}s,\\
    & \text{such that }\; \dot{y}_{x}^{\alpha}(s) = \alpha(s),\quad y_{x}^{\alpha}(0)=x\in\mathbb{R}^n\\
    & \text{and the controls } \alpha(\cdot):[0,\infty) \to B_{M} \text{ are measurable. } 
\end{aligned}
\end{equation}

Let us be given $\varepsilon>0$ fixed, and consider an $\varepsilon$-optimal trajectory $y_{x}^{\alpha_{\varepsilon}}(\cdot)$ for the problem \eqref{OCP lambda}, as in \textit{Case 2} of Definition \ref{def: quasi}. Recall the occupation measure
\begin{equation*}
    \mu^{\lambda}[x,\alpha_{\varepsilon}](\delta) = \Gamma^{\lambda}[x,\alpha_{\varepsilon}](K_{\delta})\; \in [0,1]
\end{equation*}
where $K_{\delta}$ is the set of quasi-minimizers defined in \eqref{eq: set K_delta}, and $\Gamma^{\lambda}[x,\alpha](\cdot)$ is in \textit{Case 2} of Definition \ref{def: occ}.

The following results can be established either by adapting the arguments from the previous subsection, mutatis mutandis, or by appealing to the result in the preliminary section (in particular \eqref{eq: conv mu} and Lemma \ref{lem: conv measures}) in conjunction with the results in the previous subsection. 

\begin{theorem}\label{thm: stab lambda}
Let Assumptions \ref{f: nice} and \ref{f: has min} hold. 
Let $\varepsilon>0$ be given, and $y_{x}^{\alpha_{\varepsilon}}(\cdot)$ be an $\varepsilon$-optimal trajectory for \eqref{OCP lambda} starting from $x$. Define $R:= \sqrt{6\|f\|_{\infty}}$. Then the corresponding occupational measure satisfies
\begin{equation*}
    \mu^{\lambda}[x,\alpha_{\varepsilon}](\delta) \leq \lambda\left(\frac{R \operatorname{dist}(x,\mathfrak{M}) + \varepsilon}{\delta}\right) \quad\quad \forall\; \delta>0.
\end{equation*}
In particular, when $\varepsilon=0$, i.e. for an optimal trajectory $y_{x}^{\alpha_{*}}(\cdot)$, one obtains 
\begin{equation*}
    \mu^{\lambda}[x,\alpha_{*}](\delta) \leq \lambda\,\frac{R \operatorname{dist}(x,\mathfrak{M})}{\delta} \quad\quad \forall\; \delta>0.
\end{equation*}
\end{theorem}

This theorem shows that the fraction of time that a (quasi-)optimal trajectory spends away from the set of global minimizers is $\mathcal{O}(\lambda)$.

\begin{corollary}\label{cor: stab lambda}
In the situation of Theorem \ref{thm: stab lambda}, if additionally assumption \ref{f: stab} holds, then for all $\eta>0$, there exist a small $\lambda>0$, and some $s\in [0,\infty)$ such that
\begin{equation*}
    \operatorname{dist}\left(y_{x}^{\alpha_{\varepsilon}}(s),\mathfrak{M}\right)\leq \eta.
\end{equation*}
The same result is satisfied by an optimal trajectory.
\end{corollary}

\subsection{The finite-horizon un-discounted problem}

We now treat the case where $t<\infty$ and $\lambda=0$ in \eqref{opt intro}. The resulting optimal control problem is then
\begin{equation}\label{OCP time}
\begin{aligned}
    u(x,t) = \; 
    \inf\limits_{\alpha(\cdot)} \; & \int_{0}^{t}\; \frac{1}{2}|\alpha(s)|^{2} + f(y_{x}^{\alpha}(s)) \;\text{d}s,\\
    & \text{such that }\; \dot{y}_{x}^{\alpha}(s) = \alpha(s),\quad y_{x}^{\alpha}(0)=x\in\mathbb{R}^n\\
    & \text{and the controls } \alpha(\cdot):[0,\infty) \to B_{M} \text{ are measurable. } 
\end{aligned}
\end{equation}

Let us be given $\varepsilon>0$ fixed, and consider an $\varepsilon$-optimal trajectory $y_{x}^{\alpha_{\varepsilon}}(\cdot)$ for the problem \eqref{OCP time}, as in \textit{Case 3} of Definition \ref{def: quasi}. Recall the occupation measure
\begin{equation*}
    \mu_{t}[x,\alpha_{\varepsilon}](\delta) = \Gamma_{t}[x,\alpha_{\varepsilon}](K_{\delta})\; \in [0,1]
\end{equation*}
where $K_{\delta}$ is the set of quasi-minimizers defined in \eqref{eq: set K_delta}, and $\Gamma_{t}[x,\alpha](\cdot)$ is in \textit{Case 3} of Definition \ref{def: occ}.

Recalling again \eqref{eq: conv mu}, the following results can be obtained as a consequence of Theorem \ref{thm: stab lambda t} and Corollary \ref{cor: stab lambda t}, or directly by adapting their proofs to the present setting. 

\begin{theorem}\label{thm: stab t}
Let Assumptions \ref{f: nice} and \ref{f: has min} hold. 
Let $\varepsilon>0$ be given, and $y_{x}^{\alpha_{\varepsilon}}(\cdot)$ be an $\varepsilon$-optimal trajectory for \eqref{OCP time} starting from $x$. Define $R:= \sqrt{6\|f\|_{\infty}}$. Then the corresponding occupational measure satisfies
\begin{equation*}
    \mu_{t}[x,\alpha_{\varepsilon}](\delta) \leq \frac{1}{t}\left(\frac{R \operatorname{dist}(x,\mathfrak{M}) + \varepsilon}{\delta}\right) \quad\quad \forall\; \delta>0.
\end{equation*}
In particular, when $\varepsilon=0$, i.e. for an optimal trajectory $y_{x}^{\alpha_{*}}(\cdot)$, one obtains 
\begin{equation*}
    \mu_{t}[x,\alpha_{*}](\delta) \leq \frac{1}{t}\,\frac{R \operatorname{dist}(x,\mathfrak{M})}{\delta} \quad\quad \forall\; \delta>0.
\end{equation*}
\end{theorem}

The latter result shows that the fraction of time that a (quasi-)optimal trajectory spends away from the set of global minimizers is $\mathcal{O}(1/t)$.

\begin{corollary}\label{cor: stab t}
In the situation of Theorem \ref{thm: stab t}, if additionally assumption \ref{f: stab} holds, then for all $\eta>0$, there exist a large $t>0$, and some $s\in [0,t]$ such that
\begin{equation*}
    \operatorname{dist}\left(y_{x}^{\alpha_{\varepsilon}}(s),\mathfrak{M}\right)\leq \eta.
\end{equation*}
The same result is satisfied by an optimal trajectory.
\end{corollary}

\section{Stability of the set of global minimizers}\label{sec: stab}

This section concerns Lyapunov stability of (quasi-)optimal trajectories. Mainly, we would like to show that when a trajectory starts close to the set $\mathfrak{M}$ of global minimizers, it remains as such forever. 

\subsection{The finite-horizon discounted problem}

Recall the optimal control problem \eqref{OCP time lambda M}. 

\begin{lemma}\label{lem: lower bound mu 1}
Let Assumptions \textbf{(A)} be satisfied. 
Let $y_{x}^{\alpha_{\varepsilon}}(\cdot)$ be an $\varepsilon$-optimal trajectory for the control problem \eqref{OCP time lambda M}, and $\delta>0$ fixed. If there exists $\tau \in \,]0,t[$ such that $\operatorname{dist}\left(y_{x}^{\alpha_{\varepsilon}}(\tau), \mathfrak{M}\right) > \delta$, then 
\begin{equation*}
    \mu_{t}^{\lambda}[x,\alpha_{\varepsilon}]\big(\gamma(\delta/2)\big) \geq \lambda \frac{e^{-\lambda \tau}}{1-e^{-\lambda t}}\frac{\delta}{M} \quad \quad \forall\, \lambda>0,\, t\in \;]0,\infty[
\end{equation*}
where $\gamma(\cdot)$ is the one in Assumption \ref{f: stab}.
\end{lemma}

\begin{proof}
Let $\delta>0$ be fixed. Suppose $\exists\,\tau \in \;]0,t[$ such that $\operatorname{dist}\left(y_{x}^{\alpha_{\varepsilon}}(\tau), \mathfrak{M}\right) > \delta$. Using the triangular inequality, one gets
\begin{equation*}
    \delta 
     < \operatorname{dist}\left(y_{x}^{\alpha_{\varepsilon}}(\tau), \mathfrak{M}\right) 
    \leq \operatorname{dist}\left(y_{x}^{\alpha_{\varepsilon}}(s), \mathfrak{M}\right) + M|s-\tau| \quad \quad \forall\, s\in [0,t]
\end{equation*}
where in the second inequality we used Lipschitz continuity of the trajectories, consequence of $|\dot{y}_{x}^{\alpha_{\varepsilon}}(s)| = |\alpha_{\varepsilon}(x)| \leq M$ and Remark \ref{rem: bounded controls}. In particular one gets
\begin{equation*}
    \delta - M|s-\tau| \leq \operatorname{dist}\left(y_{x}^{\alpha_{\varepsilon}}(s), \mathfrak{M}\right).
\end{equation*}
We wish to have $\frac{\delta}{2} < \delta - M|s-\tau|$. Thus we restrict $s$ to the interval 
\begin{equation*}
     \mathcal{I}:= \,]\tau\,-\,\delta/(2M) , \;\tau\,+\,\delta/(2M)[ \;\cap \; [0,t]
\end{equation*}
and obtain 
\begin{equation*}
    \operatorname{dist}\left(y_{x}^{\alpha_{\varepsilon}}(s), \mathfrak{M}\right) > \frac{\delta}{2} \quad \quad \forall\, s\in \mathcal{I}.
\end{equation*}

Note that, provided we choose $M$ in Remark \ref{rem: bounded controls} large enough, and if we let the time horizon $t$ be large, one can assume without loss of generality that $0\,<\, \tau -\delta/(2M)$ and $\tau +\delta/(2M)\,<\,t$ so that $\;\mathcal{I} =\; ]\tau\,-\,\delta/(2M) , \;\tau\,+\,\delta/(2M)[$.

Thus one obtains
\begin{equation*}
    f(y_{x}^{\alpha_{\varepsilon}}(s)) \geq \inf\limits_{z\in \mathbb{R}^{n}} \left\{f(z) \,:\, \operatorname{dist}\left(z, \mathfrak{M}\right) > \frac{\delta}{2} \right\}\quad \quad \forall\, s\in \mathcal{I}.
\end{equation*}
Subtracting $\underline{f}$ from both sides of the inequality, and recalling the function $\gamma(\cdot)$ in Assumption \ref{f: stab}, one gets
\begin{equation*}
    f(y_{x}^{\alpha_{\varepsilon}}(s)) - \underline{f} \; > \; \gamma(\delta/2)\quad \quad \forall\, s\in \mathcal{I}.
\end{equation*}
Recalling from \eqref{eq: set K_delta} the definition of the set $K^{c}_{\gamma(\delta/2)} = \{ z\in \mathbb{R}^{n} \,:\, f(z) - \underline{f} > \gamma(\delta/2)\}$, the latter inequality shows that
\begin{equation*}
    y_{x}^{\alpha_{\varepsilon}}(s) \in K^{c}_{\gamma(\delta/2)}\quad \quad \forall\, s\in \mathcal{I}.
\end{equation*}
We can now estimate the following (see the notations and definition in \S \ref{sec: occ})
\begin{equation}\label{eq: lower bound on mu in proof}
\begin{aligned}
     \mu_{t}^{\lambda}[x,\alpha_{\varepsilon}]\big(\gamma(\delta/2)\big)
     & =  \frac{\lambda}{1-e^{-\lambda t}} \int_{0}^{t} \mathds{1}_{K^{c}_{\gamma(\delta/2)}} \big(y_{x}^{\alpha_{\varepsilon}}(s)\big)\,e^{-\lambda s}\,\dd s\\
     & \geq \frac{\lambda}{1-e^{-\lambda t}} \int_{\mathcal{I}} e^{-\lambda s}\,\dd s = \frac{2}{1-e^{-\lambda t}} e^{-\lambda \tau} \sinh\left(\frac{\lambda\delta}{2M}\right)
\end{aligned}
\end{equation}
where $\sinh(\cdot)$ is the hyperbolic sinus function. The conclusion then follows noting that $\sinh(z)\geq z$ when $z>0$.  
\end{proof}

\begin{theorem}\label{thm: Lyapunov lambda t}
Let Assumptions \textbf{(A)} be satisfied. 
Let $y_{x}^{\alpha_{\varepsilon}}(\cdot)$ be an $\varepsilon$-optimal trajectory for problem (\ref{OCP time lambda M}). 
Then when $\varepsilon$ is small, the set of global minimizers $\mathfrak{M}$ is Lyapunov stable, i.e. 
\begin{equation*}
    \forall \; \delta>0,\; \exists \;\eta>0 \text{ such that } \operatorname{dist}(x, \mathfrak{M}) \leq \eta \;\Rightarrow\; \operatorname{dist}\left(y_{x}^{\alpha_{\varepsilon}}(s), \mathfrak{M}\right) \leq \delta \quad \forall s\in \; [0,t].
\end{equation*}
This is in particular satisfied by an optimal trajectory.
\end{theorem}

\begin{proof}
We proceed by contradiction. Let $\delta>0$ be fixed, and suppose that for all $\eta>0$ there exists $\tau\in \;]0,t[$ such that $\operatorname{dist}\left(y_{x}^{\alpha_{\varepsilon}}(\tau), \mathfrak{M}\right) > \delta$ and $\operatorname{dist}(x, \mathfrak{M}) \leq \eta$. Then Lemma \ref{lem: lower bound mu 1} ensures that 
\begin{equation*}
    \lambda \frac{e^{-\lambda \tau}}{1-e^{-\lambda t}}\frac{\delta}{M} \leq \mu_{t}^{\lambda}[x,\alpha_{\varepsilon}]\big(\gamma(\delta/2)\big).
\end{equation*}
On the other hand, Theorem \ref{thm: stab lambda t} guarantees 
\begin{equation*}
    \mu_{t}^{\lambda}[x,\alpha_{\varepsilon}](\gamma(\delta/2)) \leq \frac{\lambda}{1-e^{-\lambda t}}\left(\frac{R \operatorname{dist}(x,\mathfrak{M}) + \varepsilon}{\gamma(\delta/2)}\right).
\end{equation*}
Therefore, and since $\operatorname{dist}(x, \mathfrak{M}) \leq \eta$, one gets
\begin{equation*}
     e^{-\lambda \tau}\frac{\delta}{M} 
     \leq \frac{\eta\,R  + \varepsilon}{\gamma(\delta/2)} \quad \quad \text{ that is } \quad \gamma(\delta/2)\frac{\delta}{M} 
     \leq (\eta\,R  + \varepsilon) e^{\lambda \tau}.
\end{equation*}
However, choosing $\eta$ and $\varepsilon$ small leads to a contradiction, and thus concludes the proof.
\end{proof}

\subsection{The infinite-horizon discounted problem}

Recall the optimal control problem \eqref{OCP lambda}.

\begin{lemma}\label{lem: lower bound mu 2}
Let Assumptions \textbf{(A)} be satisfied. 
Let $y_{x}^{\alpha_{\varepsilon}}(\cdot)$ be an $\varepsilon$-optimal trajectory for the control problem \eqref{OCP lambda}, and $\delta>0$ fixed. If  $\,\exists\,\tau \in \,]0,\infty[$ such that $\operatorname{dist}\left(y_{x}^{\alpha_{\varepsilon}}(\tau), \mathfrak{M}\right) > \delta$, then 
\begin{equation*}
    \mu^{\lambda}[x,\alpha_{\varepsilon}]\big(\gamma(\delta/2)\big) \geq \lambda \frac{\delta}{M} e^{-\lambda \tau} \quad \quad \forall\, \lambda>0
\end{equation*}
where $\gamma(\cdot)$ is the one in Assumption \ref{f: stab}.
\end{lemma}

\begin{proof}
The proof is the same as for Lemma \ref{lem: lower bound mu 1}, the only difference being the lower bound in \eqref{eq: lower bound on mu in proof} which becomes
\begin{equation*}
\begin{aligned}
     \mu^{\lambda}[x,\alpha_{\varepsilon}]\big(\gamma(\delta/2)\big)
     & =  \lambda \int_{0}^{\infty} \mathds{1}_{K^{c}_{\gamma(\delta/2)}} \big(y_{x}^{\alpha_{\varepsilon}}(s)\big)\,e^{-\lambda s}\,\dd s\\
     & \geq \lambda \int_{\mathcal{I}} e^{-\lambda s}\,\dd s = 2 e^{-\lambda \tau} \sinh\left(\frac{\lambda\delta}{2M}\right)\geq \lambda \frac{\delta}{M}e^{-\lambda \tau}
\end{aligned}
\end{equation*}
where we recall $\;\mathcal{I} =\; ]\tau\,-\,\delta/(2M) , \;\tau\,+\,\delta/(2M)[$.
\end{proof}

An alternative proof could be to invoke the convergence in Lemma \ref{lem: conv measures}, see also \eqref{eq: conv mu}, and send $t\to\infty$ in the result of Lemma \ref{lem: lower bound mu 1}.

\begin{theorem}\label{thm: Lyapunov lambda}
Let Assumptions \textbf{(A)} be satisfied. 
Let $y_{x}^{\alpha_{\varepsilon}}(\cdot)$ be an $\varepsilon$-optimal trajectory for problem (\ref{OCP lambda}). Then when $\varepsilon$ is small, the set of global minimizers $\mathfrak{M}$ is Lyapunov stable, i.e.
\begin{equation*}
    \forall \; \delta>0,\; \exists \;\eta>0 \text{ such that } \operatorname{dist}(x, \mathfrak{M}) \leq \eta \;\Rightarrow\; \operatorname{dist}\left(y_{x}^{\alpha_{\varepsilon}}(s), \mathfrak{M}\right) \leq \delta \quad \forall s\in \; ]0,\infty[.
\end{equation*}
This is in particular satisfied by an optimal trajectory.
\end{theorem}

\begin{proof}
We proceed by contradiction. Let $\delta>0$ be fixed, and suppose that for all $\eta>0$ there exists $s\in \;]0,t[$ such that $\operatorname{dist}\left(y_{x}^{\alpha_{\varepsilon}}(s), \mathfrak{M}\right) > \delta$ and $\operatorname{dist}(x, \mathfrak{M}) \leq \eta$. Then Lemma \ref{lem: lower bound mu 2} ensures that 
\begin{equation*}
    \lambda \frac{\delta}{M} e^{-\lambda \tau} \leq \mu^{\lambda}[x,\alpha_{\varepsilon}]\big(\gamma(\delta/2)\big).
\end{equation*}
On the other hand, Theorem \ref{thm: stab lambda} guarantees 
\begin{equation*}
    \mu_{t}^{\lambda}[x,\alpha_{\varepsilon}](\gamma(\delta/2)) \leq \lambda\left(\frac{R \operatorname{dist}(x,\mathfrak{M}) + \varepsilon}{\gamma(\delta/2)}\right).
\end{equation*}
Therefore, and since $\operatorname{dist}(x, \mathfrak{M}) \leq \eta$, one gets
\begin{equation*}
     \frac{\delta}{M} e^{-\lambda \tau}
     \leq \frac{\eta\,R  + \varepsilon}{\gamma(\delta/2)} \quad \quad \text{ that is } \quad \gamma(\delta/2)\frac{\delta}{M} 
     \leq (\eta\,R  + \varepsilon) e^{\lambda \tau}.
\end{equation*}
However, choosing $\eta$ and $\varepsilon$ small leads to a contradiction, and thus concludes the proof.
\end{proof}

\subsection{The finite-horizon un-discounted problem}

Recall the optimal control problem \eqref{OCP time}.

\begin{lemma}\label{lem: lower bound mu 3}
Let Assumptions \textbf{(A)} be satisfied. 
Let $y_{x}^{\alpha_{\varepsilon}}(\cdot)$ be an $\varepsilon$-optimal trajectory for the control problem \eqref{OCP time}, and $\delta>0$ fixed. If  $\,\exists\,\tau \in \,]0,t[$ such that $\operatorname{dist}\left(y_{x}^{\alpha_{\varepsilon}}(\tau), \mathfrak{M}\right) > \delta$, then 
\begin{equation*}
    \mu_{t}[x,\alpha_{\varepsilon}]\big(\gamma(\delta/2)\big) \geq \frac{1}{t}\, \frac{\delta}{M} \quad \quad \forall\,  t\in \;]0,\infty[
\end{equation*}
where $\gamma(\cdot)$ is the one in Assumption \ref{f: stab}.
\end{lemma}

\begin{proof}
The proof is the same as for Lemma \ref{lem: lower bound mu 1}, the only difference being the lower bound in \eqref{eq: lower bound on mu in proof} which becomes
\begin{equation*}
\begin{aligned}
     \mu_{t}[x,\alpha_{\varepsilon}]\big(\gamma(\delta/2)\big)
     & =  \frac{1}{t} \int_{0}^{t} \mathds{1}_{K^{c}_{\gamma(\delta/2)}} \big(y_{x}^{\alpha_{\varepsilon}}(s)\big)\,\dd s
     \geq \frac{1}{t}\int_{\mathcal{I}} 1\,\dd s = \frac{1}{t}\,\frac{\delta}{M}
\end{aligned}
\end{equation*}
where $\;\mathcal{I} =\; ]\tau\,-\,\delta/(2M) , \;\tau\,+\,\delta/(2M)[$.
\end{proof}

An alternative proof could be to invoke the convergence in Lemma \ref{lem: conv measures}, see also \eqref{eq: conv mu}, and send $\lambda\to 0$ in the result of Lemma \ref{lem: lower bound mu 1}.

\begin{theorem}\label{thm: Lyapunov t}
Let Assumptions \textbf{(A)} be satisfied. 
Let $y_{x}^{\alpha_{\varepsilon}}(\cdot)$ be an $\varepsilon$-optimal trajectory for problem (\ref{OCP time}). Then when $\varepsilon$ is small, the set of global minimizers $\mathfrak{M}$ is Lyapunov stable, i.e. 
\begin{equation*}
    \forall \; \delta>0,\; \exists \;\eta>0 \text{ such that } \operatorname{dist}(x, \mathfrak{M}) \leq \eta \;\Rightarrow\; \operatorname{dist}\left(y_{x}^{\alpha_{\varepsilon}}(s), \mathfrak{M}\right) \leq \delta \quad \forall s\in \; [0,t].
\end{equation*}
This is in particular satisfied by an optimal trajectory.
\end{theorem}

\begin{proof}
We proceed by contradiction. Let $\delta>0$ be fixed, and suppose that for all $\eta>0$ there exists $s\in \;]0,t[$ such that $\operatorname{dist}\left(y_{x}^{\alpha_{\varepsilon}}(s), \mathfrak{M}\right) > \delta$ and $\operatorname{dist}(x, \mathfrak{M}) \leq \eta$. Then Lemma \ref{lem: lower bound mu 3} ensures that 
\begin{equation*}
    \frac{1}{t}\, \frac{\delta}{M} \leq 
    \mu_{t}[x,\alpha_{\varepsilon}]\big(\gamma(\delta/2)\big).
\end{equation*}
On the other hand, Theorem \ref{thm: stab t} guarantees 
\begin{equation*}
    \mu_{t}[x,\alpha_{\varepsilon}]\big(\gamma(\delta/2)\big) \leq \frac{1}{t}\left(\frac{R \operatorname{dist}(x,\mathfrak{M}) + \varepsilon}{\gamma(\delta/2)}\right).
\end{equation*}
Therefore, and since $\operatorname{dist}(x, \mathfrak{M}) \leq \eta$, one gets
\begin{equation*}
     \frac{\delta}{M} 
     \leq \frac{\eta\,R  + \varepsilon}{\gamma(\delta/2)} \quad \quad \text{ that is } \quad \gamma(\delta/2)\frac{\delta}{M} 
     \leq (\eta\,R  + \varepsilon) .
\end{equation*}
However, choosing $\eta$ and $\varepsilon$ small leads to a contradiction, and thus concludes the proof.
\end{proof}

\section{Conclusion: The convergence}\label{sec: conclusion}

Our main objective is to design trajectories which solve the optimization problems
\begin{equation*}
    \min\limits_{z\in\mathbb{R}^{n}}\, f(z).
\end{equation*}
To this end, we have considered the optimal control problem
\begin{equation}\label{OCP conclusion}
\begin{aligned}
    u_{\lambda}(x,t) = \; 
    \inf\limits_{\alpha(\cdot)} \; & \int_{0}^{t}\; \left(\frac{1}{2}|\alpha(s)|^{2} + f(y_{x}^{\alpha}(s))\right)\, e^{-\lambda s}\;\text{d}s,\\
    & \text{such that }\; \dot{y}_{x}^{\alpha}(s) = \alpha(s),\quad y_{x}^{\alpha}(0)=x\in\mathbb{R}^n\\
    & \text{and the controls } \alpha(\cdot):[0,\infty) \to B_{M} \text{ are measurable. } 
\end{aligned}
\end{equation}
and studied dynamical properties of its quasi-optimal trajectories, regardless of the initial position $x$, in three cases:
\begin{enumerate}[label = \underline{\textit{Case \arabic*}}.]
    \item When $t<\infty$, and $\lambda>0$: this is a consequence of Corollary \ref{cor: stab lambda t} and Theorem \ref{thm: Lyapunov lambda t}, when $\varepsilon$ and $\lambda$ are small, and $t$ is large. 
    \item When $t=\infty$, and $\lambda>0$: this is a consequence of Corollary \ref{cor: stab lambda} and Theorem \ref{thm: Lyapunov lambda}, when $\varepsilon$ and $\lambda$ are small.
    \item When $t<\infty$, and $\lambda=0$:  combine Corollary \ref{cor: stab t} and Theorem \ref{thm: Lyapunov t}, when $\varepsilon$ is small, and $t$ is large.  
\end{enumerate}

The following theorem  summarizes the global practical asymptotic convergence towards the set of global minimizers, as a consequence of the previous results.

\begin{theorem}\label{thm: main conclusion}
Let Assumptions \ref{f: nice}, \ref{f: has min}, and \ref{f: stab} be satisfied. 
Let $y_{x}^{\alpha_{\varepsilon}}(\cdot)$ be an $\varepsilon$-optimal trajectory for problem (\ref{OCP time lambda M}) with $\varepsilon\geq 0$ small. 
Then for all $\eta>0$, there exist $\lambda>0$ small, $t>0$ large, and $\tau = \tau(\eta,\varepsilon, \lambda, t)>0$ such that 
\begin{equation*}
    \operatorname{dist}\left(y_{x}^{\alpha_{\varepsilon}}(s), \mathfrak{M}\right) \leq \eta \quad \quad \forall\, s\in [\tau, t].
\end{equation*}
In other words, $y_{x}^{\alpha_{\varepsilon}}(s) \in \left\{z\in \mathbb{R}^{n}: \operatorname{dist}(z,\mathfrak{M})\leq \eta\right\}$ for all $s\in [\tau, t]$. 

The same conclusion holds for the problem \eqref{OCP lambda}, i.e. when $\lambda>0$ and $t=\infty$ in \eqref{OCP conclusion}, and for the problem \eqref{OCP time}, i.e. when $\lambda = 0$ and $t<\infty$ in \eqref{OCP conclusion}.
\end{theorem}

\begin{proof}
We make the proof for the problem \eqref{OCP conclusion} in \textit{Case 1}. In the remaining two cases, the arguments are parallel. 

Let $\eta>0$ be arbitrarily fixed. From Theorem \ref{thm: Lyapunov lambda t}, we know that $\exists\, \delta>0$ such that: if $\;\operatorname{dist}(x, \mathfrak{M}) \leq \delta\;$, then $\; \operatorname{dist}\left(y_{x}^{\alpha_{\varepsilon}}(s), \mathfrak{M}\right) \leq \eta \quad \forall s\in \; [0,t] \;$. This holds for any initial position $x = y_{x}^{\alpha_{\varepsilon}}(0)\in \mathbb{R}^{n}$. 

From Corollary \ref{cor: stab lambda t}, we know in particular that for the previously chosen $\delta>0$, there exist a large $t>0$, a small $\lambda>0$, and some $\tau \in [0,t]$ such that $\;\operatorname{dist}\left(y_{x}^{\alpha_{\varepsilon}}(\tau),\mathfrak{M}\right)\leq \delta.\;$ Therefore, consider the trajectory
\begin{equation*}
    \overline{y}_{x}^{\alpha_{\varepsilon}}(\cdot) := y_{x}^{\alpha_{\varepsilon}}(\tau + \cdot) \quad \quad \text{in } [0,t-\tau].
\end{equation*}
Its initial value satisfies $\;\operatorname{dist}(\overline{y}_{x}^{\alpha_{\varepsilon}}(0), \mathfrak{M}) \leq \delta\;$, and hence by virtue of Theorem \ref{thm: Lyapunov lambda t} one has $\; \operatorname{dist}\left(\overline{y}_{x}^{\alpha_{\varepsilon}}(s), \mathfrak{M}\right) \leq \eta \quad \forall s\in \; ]0,t[ \;$, which means  $\; \operatorname{dist}\left(y_{x}^{\alpha_{\varepsilon}}(s), \mathfrak{M}\right) \leq \eta \quad \forall s\in \; [\tau,t]. \;$
\end{proof}

\appendix

\section{Proof of Lemma \ref{lem: bounds}}\label{appendix}

We will make the proof for $u_{\lambda}(x,t)$ which we recall is the value function of \eqref{opt intro} and the viscosity solution of \eqref{eq: HJB lambda t}.  The arguments are parallel for the other value functions $u_{\lambda}(x)$ corresponding to $(t=\infty, \lambda>0)$ in \eqref{opt intro} and $u(x,t)$ corresponding to $(t<\infty, \lambda=0)$ in \eqref{opt intro}, and which are the viscosity solutions to \eqref{eq: HJB lambda} and \eqref{eq: HJB t} respectively.

\begin{proof}[Proof of (i)] From the boundedness of $f(\cdot)$ in assumption \ref{f: nice}, we have $-\|f\|_{\infty} \leq f(x)$ for all $x\in \mathbb{R}^{n}$. Hence for any $\alpha(\cdot)$, it holds 
\begin{equation}
\label{eq: lower bound of u in proof}
    -\|f\|_{\infty} \frac{1-e^{-\lambda t}}{\lambda} \leq  \int_{0}^{t}\; \left(\frac{1}{2}|\alpha(s)|^{2} + f(y_{x}^{\alpha}(s))\right)\, e^{-\lambda s}\;\text{d}s.
\end{equation}
Taking the infimum over $\alpha(\cdot)$ yields $-\|f\|_{\infty} \frac{1-e^{-\lambda t}}{\lambda} \leq u_{\lambda}(x,t)$. \\
To get the upper-bound, consider the admissible control $\alpha(\cdot)\equiv 0$. Thus we have 
\begin{equation}
    \label{eq: nice bound on u in proof}
    u_{\lambda}(x,t) \leq f(x) \;\frac{1-e^{-\lambda t}}{\lambda}
\end{equation}
which is upper-bounded by $\|f\|_{\infty} \frac{1-e^{-\lambda t}}{\lambda}$. This shows that
\begin{equation}\label{bound in proof of i}
    |u_{\lambda}(x,t)|\leq \|f\|_{\infty} \frac{1-e^{-\lambda t}}{\lambda}
\end{equation}
hence $|\lambda\,u_{\lambda}(x,t)| \leq \|f\|_{\infty}$.
\end{proof}

\begin{proof}[Proof of (ii)]
We will construct viscosity sub-/super-solutions to \eqref{eq: HJB lambda t}, and conclude with a comparison principle. \\
Consider the function $\overline{v}(x,t) := u_{\lambda}(x,t+h) + \frac{1-e^{-\lambda h}}{\lambda}\|f\|_{\infty}$ where $h>0$ is fixed. Since $u_{\lambda}(x,t)$ is a viscosity solution to \eqref{eq: HJB lambda t}, the function $\overline{v}(x,t)$ solves the equation
\begin{equation*}
\left\{\;
\begin{aligned}
    \partial_{t} \overline{v}(x,t) + \lambda\,\overline{v}(x,t) + \frac{1}{2}|D\overline{v}(x,t)|^{2} &= f(x) + (1-e^{-\lambda h})\|f\|_{\infty},  &&\quad(x,t) \in \mathbb{R}^{n}\times(0,+\infty)\\
    \overline{v}(x,0) &= u_{\lambda}(x,h) + \frac{1-e^{-\lambda h}}{\lambda}\|f\|_{\infty},  &&\quad x\in \mathbb{R}^{n}.
\end{aligned}
\right.
\end{equation*}
Since $f(x) + (1-e^{-\lambda h})\|f\|_{\infty}\geq f(x)$, and $u_{\lambda}(x,h) + \frac{1-e^{-\lambda h}}{\lambda}\|f\|_{\infty}\geq 0$ thanks to \eqref{bound in proof of i}, the function $\overline{v}(x,t)$ is a viscosity super-solution to \eqref{eq: HJB lambda t}, while $u_{\lambda}(x,t)$ is a viscosity solution and consequently also a viscosity sub-solution to \eqref{eq: HJB lambda t}. We can now conclude by invoking the comparison principle in \cite[Theorem 2.1]{da2006uniqueness}\footnote{Note that the result in \cite{da2006uniqueness} concerns Hamiltonians which do not depend on zero-order terms (mainly the term $\lambda u_{\lambda}(x,t)$ in our case). However, the PDE \eqref{eq: HJB lambda t} can be rewritten using the exponential (Kruzkov) transform $\omega_{\lambda}(x,t) := e^{\lambda t}u_{\lambda}(x,t)$. Then the PDE satisfied by $\omega_{\lambda}$ takes the form $\partial_{t}\omega_{\lambda} + \frac{e^{-\lambda t}}{2}|D\omega_{\lambda}|^{2} = e^{\lambda t}f(x)$. The corresponding Hamiltonian is  $\frac{e^{-\lambda t}}{2}|D\omega_{\lambda}|^{2} - e^{\lambda t}f(x) = H(x,t,D\omega_{\lambda}) = \sup\limits_{\alpha}\left\{-\alpha\cdot D\omega_{\lambda} - \frac{e^{\lambda t}}{2}|\alpha^{2}| - e^{\lambda t}f(x)\right\}$ which fits to the one studied in \cite{da2006uniqueness}.} and obtain $u_{\lambda}(x,t) \leq \overline{v}(x,t)$.  \\
Consider now the function $\underline{v}(x,t) := u_{\lambda}(x,t+h) - \frac{1-e^{-\lambda h}}{\lambda}\|f\|_{\infty}$. As we did for $\overline{v}(x,t)$, we obtain 
\begin{equation*}
\left\{\;
\begin{aligned}
    \partial_{t} \underline{v}(x,t) + \lambda\,\underline{v}(x,t) + \frac{1}{2}|D\underline{v}(x,t)|^{2} &\leq f(x),  &&\quad(x,t) \in \mathbb{R}^{n}\times(0,+\infty)\\
    \overline{v}(x,0) &\leq 0,  &&\quad x\in \mathbb{R}^{n}
\end{aligned}
\right.
\end{equation*}
that is $\underline{v}(x,t)$ is a viscosity sub-solution to \eqref{eq: HJB lambda t}. Again since $u_{\lambda}(x,t)$ is a viscosity solution it is also a viscosity super-solution to \eqref{eq: HJB lambda t}, and we conclude with the same comparison principle that $\underline{v}(x,t) \leq u_{\lambda}(x,t)$. All together, we obtain
\begin{equation*}
    |u_{\lambda}(x,t) - u_{\lambda}(x,t+h)| \leq  \frac{1-e^{-\lambda h}}{\lambda}\|f\|_{\infty} \leq |h|\, \|f\|_{\infty}
\end{equation*}
which is the desired estimate. 
\end{proof}

\begin{proof}[Proof of (iii)]
We divide the proof into three steps. First, we approximate the right-hand side $f(\cdot)$ by a family of smooth (Lipschitz continuous) functions using standard mollification. The associated optimal control problem yields a value function whose gradient can be estimated in terms of the approximation. In the second step, we use the corresponding HJB equation to show that this gradient bound is in fact uniform, i.e., independent of the particular approximation. Finally, we verify the uniform convergence of the value functions and pass to the limit in the approximation. By the stability of viscosity solutions, this allows us to recover the original HJB equation and the corresponding value function.

\textit{Step 1. (The approximation)} We use a standard mollification. Let $\varrho\in C^{\infty}_{c}(\mathbb{R}^{n})$, the set of smooth and compactly supported functions. Suppose $\varrho\geq 0$, its support is such that $\operatorname{supp}(\varrho)\subset B(0,1)$, the closed ball centered in $0$ with radius $1$, and $\int_{\mathbb{R}^{n}}\varrho(x)\dd x= 1$. We define the scaled mollifier $\varrho^{\varepsilon}(x) = \frac{1}{\varepsilon^{n}}\varrho(x/\varepsilon)$, and the mollified function 
\begin{equation*}
    f^{\varepsilon}(x) : = f \ast  \varrho^{\varepsilon} (x) = \int_{\mathbb{R}^{n}}f(x-y) \varrho^{\varepsilon}(y)\dd y.
\end{equation*}
Recalling Assumption \ref{f: nice}, we know  that $f^{\varepsilon}(\cdot) \in C^{\infty}(\mathbb{R}^{n})$ and $f^{\varepsilon} \to f$ uniformly in every compact subsets of $\mathbb{R}^{n}$, when $\varepsilon\to 0$. See for example \cite[Theorem 7 in \S C.4]{evans2010partial} or \cite[Lemma 7.1]{salsa2016partial}. Moreover it holds $\|f^{\varepsilon}\|_{\infty}\leq \|f\|_{\infty}$, and 
\begin{equation*}
    |D f^{\varepsilon}(x)| \leq \frac{C}{\varepsilon}\|f\|_{\infty} \quad \quad \text{ where } \; C=\int_{\mathbb{R}^{n}} D\varrho(y)\dd y.
\end{equation*}
Given such approximation $f^{\varepsilon}(\cdot)$, we consider its corresponding optimal control problem \eqref{opt intro} and associated HJB equation \eqref{eq: HJB lambda t}, which are
\begin{equation}\label{OCP eps}
\begin{aligned}
    u^{\varepsilon}_{\lambda}(x,t) := \; 
    \inf\limits_{\alpha(\cdot)} \; & \int_{0}^{t}\; \left(\frac{1}{2}|\alpha(s)|^{2} + f^{\varepsilon}(y_{x}^{\alpha}(s))\right)\, e^{-\lambda s}\;\text{d}s,\\
    & \text{such that }\; \dot{y}_{x}^{\alpha}(s) = \alpha(s),\quad y_{x}^{\alpha}(0)=x\in\mathbb{R}^n\\
    & \text{and the controls } \alpha(\cdot):[0,\infty) \to \mathbb{R}^{n} \text{ are measurable. } 
\end{aligned}
\end{equation}
and 
\begin{equation}\label{eq: HJB eps}
\left\{\;
\begin{aligned}
    \partial_{t}u^{\varepsilon}_{\lambda}(x,t) + \lambda\,u^{\varepsilon}_{\lambda}(x,t) + \frac{1}{2}|Du^{\varepsilon}_{\lambda}(x,t)|^{2} &= f^{\varepsilon}(x),  &&\quad\quad (x,t) \in \mathbb{R}^{n}\times(0,+\infty)\\
    u^{\varepsilon}_{\lambda}(x,0) &= 0,  &&\quad\quad x\in \mathbb{R}^{n}.
\end{aligned}
\right.
\end{equation}
Using \eqref{OCP eps}, we provide next a first estimate for $D u^{\varepsilon}(x,t)$. 

Let $h\in \mathbb{R}^{n}$ be fixed. Let $\alpha_{\delta}(\cdot)$ be a $\delta$-optimal control for the problem starting from $x+h$, that is
\begin{equation*}
    \int_{0}^{t}\; \left(\frac{1}{2}|\alpha_{\delta}(s)|^{2} + f^{\varepsilon}(y_{x+h}^{\alpha_{\delta}}(s))\right)\, e^{-\lambda s}\;\text{d}s \leq u^{\varepsilon}_{\lambda}(x+h,t) + \delta.
\end{equation*}
This same control being sub-optimal for the problem starting from $x$ yields
\begin{equation*}
    u^{\varepsilon}_{\lambda}(x,t) \leq \int_{0}^{t}\; \left(\frac{1}{2}|\alpha_{\delta}(s)|^{2} + f^{\varepsilon}(y_{x}^{\alpha_{\delta}}(s))\right)\, e^{-\lambda s}\;\text{d}s.
\end{equation*}
Therefore we have
\begin{equation*}
\begin{aligned}
    u^{\varepsilon}_{\lambda}(x,t) - u^{\varepsilon}_{\lambda}(x+h,t) 
    & \leq \int_{0}^{t}\; \left(f^{\varepsilon}(y_{x}^{\alpha_{\delta}}(s)) - f^{\varepsilon}(y_{x+h}^{\alpha_{\delta}}(s))\right)\, e^{-\lambda s}\;\text{d}s \, + \delta\\
    & \leq \frac{C}{\varepsilon}\|f\|_{\infty} \int_{0}^{t}\; \left|y_{x}^{\alpha_{\delta}}(s) - y_{x+h}^{\alpha_{\delta}}(s)\right|\, e^{-\lambda s}\;\text{d}s \, + \delta\\
    & \leq |h|\,\frac{C}{\varepsilon}\|f\|_{\infty} \,\frac{1-e^{-\lambda t}}{\lambda}\, +\delta
\end{aligned}
\end{equation*}
where the last inequality is a consequence of 
 $\; y^{\alpha_{\delta}}_{x}(s)  = x + \int_{0}^{s}\alpha_{\delta}(r)\dd r\;$ and $\;y^{\alpha_{\delta}}_{x+h}(s)  = x+h + \int_{0}^{s}\alpha_{\delta}(r)\dd r.\;$
Exchanging the roles of $x$ and $x+h$, and sending $\delta \to 0$, yield
\begin{equation*}
\begin{aligned}
    |u^{\varepsilon}_{\lambda}(x,t) - u^{\varepsilon}_{\lambda}(x+h,t)| 
     \leq |h|\; \frac{C}{\varepsilon}\|f\|_{\infty}  \,\frac{1-e^{-\lambda t}}{\lambda}.
\end{aligned}
\end{equation*}

\textit{Step 2. (The uniform bound)}

The previous step guarantees that the value function $u^{\varepsilon}_{\lambda}(x,t)$ is locally Lipschitz in $x$. Hence the gradient exists almost everywhere and the HJB equation \eqref{eq: HJB eps} --of which $u^{\varepsilon}_{\lambda}(x,t)$ is the viscosity solution-- is satisfied almost everywhere \cite[Proposition II.1.9(b), p.31]{bardi1997optimal}. Using the latter PDE, together with the two previous estimates $|\partial_{t}u^{\varepsilon}_{\lambda}(x,t)|\leq \|f\|_{\infty}$ and $|\lambda\,u^{\varepsilon}_{\lambda}(x,t)|\leq \|f\|_{\infty}$, we can then obtain the desired uniform bound $|Du^{\varepsilon}_{\lambda}(x,t)| \leq \sqrt{6\|f\|_{\infty}}$ almost everywhere. The same arguments when used for the control problems whose value functions are $u^{\varepsilon}_{\lambda}(x)$ and $u^{\varepsilon}(x,t)$ will show in fact that $|Du^{\varepsilon}_{\lambda}(x)|,\, |Du^{\varepsilon}(x,t)| \leq 2\sqrt{\|f\|_{\infty}}$.

\textit{Step 3. (The conclusion via stability)}

So far we know that $|Du^{\varepsilon}_{\lambda}(x,t)| \leq \sqrt{6\|f\|_{\infty}}$ a.e., that $f^{\varepsilon} \to f$ locally uniformly, and that $u^{\varepsilon}$ solves in the viscosity sense the PDE \eqref{eq: HJB eps}. We wish to send $\varepsilon\to 0$ in order to recover our original solution $u_{\lambda}(x,t)$ solution to \eqref{eq: HJB lambda t} satisfying moreover the desired bound $|Du_{\lambda}(x,t)| \leq \sqrt{6\|f\|_{\infty}}$. This shall be a consequence of the stability property for viscosity solutions, provided $u^{\varepsilon}_{\lambda} \to u_{\lambda}$ locally uniformly. See for example \cite[Proposition II.2.2, Corollary V.1.8, Corollary V.4.27]{bardi1997optimal}, \cite[\S 2.3]{barles1994solutions}, \cite[\S6]{crandall1992user}, \cite[Lemma II.6.2]{fleming2006controlled}. We verify the latter convergence as follows
\begin{equation*}
\begin{aligned}
    u^{\varepsilon}_{\lambda}(x,t) - u^{\varepsilon}_{\lambda} 
    & \leq \sup\limits_{\alpha(\cdot)} \int_{0}^{t} \left( f^{\varepsilon}(y_{x}^{\alpha}(s)) - f(y_{x}^{\alpha}(s))\right)\,e^{-\lambda s}\dd \,s \leq \|f^{\varepsilon} - f\|_{\infty}\, \frac{1-e^{-\lambda t}}{\lambda}
\end{aligned}
\end{equation*}
where the first inequality is a consequence of ``$\inf(\dots) - \inf(\dots) \leq \sup(\dots)$'', and $\|f^{\varepsilon} - f\|_{\infty}$ should be understood as the essential supremum over a compact subset of $\mathbb{R}^{n}$ which we assume is fixed and hence do not specify here. 
Exchanging the roles of $u^{\varepsilon}$ and $u$ yields
\begin{equation*}
    \|u^{\varepsilon} - u\|_{\infty} \leq \|f^{\varepsilon} - f\|_{\infty}\, \frac{1}{\lambda}
\end{equation*}
where again, the notation $\|u^{\varepsilon} - u\|_{\infty}$ should be understood locally both in $x$ and $t$. This ensures the desired local uniform convergence and concludes the proof. 
\end{proof}

\section{Proof of Lemma \ref{lem: conv measures}}\label{appendix 2}

\begin{proof}
Recall the notations in Definition \ref{def: occ}. 
Let $\mathscr{O}$ be a fixed Borel subset of $\mathbb{R}^{n}$. We have
\begin{equation*}
\begin{aligned}
    & \Gamma_{t}^{\lambda}[x,\alpha](\mathscr{O}) - \Gamma^{\lambda}[x,\alpha](\mathscr{O}) = \frac{\lambda}{1-e^{-\lambda t}} \int_0^t \mathds{1}_{\mathscr{O}}\left(y_{x}^{\alpha}(s)\right)\, e^{-\lambda s} \dd s - \lambda \int_0^{\infty} \mathds{1}_{\mathscr{O}}\left(y_{x}^{\alpha}(s)\right)\, e^{-\lambda s} \dd s\\
    & \quad = \left(\frac{\lambda}{1-e^{-\lambda t}} - \lambda\right)\int_0^t \mathds{1}_{\mathscr{O}}\left(y_{x}^{\alpha}(s)\right)\, e^{-\lambda s} \dd s - \lambda \int_{t}^{\infty} \mathds{1}_{\mathscr{O}}\left(y_{x}^{\alpha}(s)\right)\, e^{-\lambda s} \dd s\\
    & \quad \leq \left|\frac{\lambda}{1-e^{-\lambda t}} - \lambda\right| \int_0^t \, e^{-\lambda s} \dd s + \lambda \int_{t}^{\infty} \, e^{-\lambda s} \dd s \\
    & \quad = \left|\frac{\lambda}{1-e^{-\lambda t}} - \lambda\right| \frac{1- e^{-\lambda t}}{\lambda} + e^{-\lambda t} \; = \; 2\, e^{-\lambda t}
\end{aligned}
\end{equation*}
hence $|\Gamma_{t}^{\lambda}[x,\alpha](\mathscr{O}) - \Gamma^{\lambda}[x,\alpha](\mathscr{O})| \leq 2 \, e^{-\lambda t}$ for all $\mathscr{O}$, and $\|\Gamma_{t}^{\lambda}[x,\alpha](\cdot) - \Gamma^{\lambda}[x,\alpha](\cdot)\|_{TV}\to 0$ when $t\to\infty$. 

We also have
\begin{equation*}
\begin{aligned}
    & \Gamma_{t}^{\lambda}[x,\alpha](\mathscr{O}) - \Gamma_{t}[x,\alpha](\mathscr{O}) = \frac{\lambda}{1-e^{-\lambda t}} \int_0^t \mathds{1}_{\mathscr{O}}\left(y_{x}^{\alpha}(s)\right)\, e^{-\lambda s} \dd s - \frac{1}{t} \int_0^{t} \mathds{1}_{\mathscr{O}}\left(y_{x}^{\alpha}(s)\right)\,\dd s\\
    &\quad = \left(\frac{\lambda t}{1-e^{-\lambda t}} - 1\right)  \frac{1}{t}\int_0^{t} \mathds{1}_{\mathscr{O}}\left(y_{x}^{\alpha}(s)\right)\, e^{-\lambda s}\,\dd s - \frac{1}{t}\int_0^{t} \mathds{1}_{\mathscr{O}}\left(y_{x}^{\alpha}(s)\right)\, \left(1-e^{-\lambda s}\right)\,\dd s\\
    &\quad \leq  \left|\frac{\lambda t}{1-e^{-\lambda t}} - 1\right| + \frac{1}{t} \left|\int_0^{t} \left(1-e^{-\lambda s}\right)\,\dd s\right| = \left|\frac{\lambda t}{1-e^{-\lambda t}} - 1\right| + \left|1 - \frac{1 - e^{-\lambda t}}{\lambda t}\right|.
\end{aligned}
\end{equation*}
The latter bound being independent of $\mathscr{O}$ and vanishes when $\lambda \to 0$, the conclusion then follows: $\|\Gamma_{t}^{\lambda}[x,\alpha](\cdot) - \Gamma_{t}[x,\alpha](\cdot)\|_{TV}\to 0$ when $\lambda\to 0$.
\end{proof}

\bibliographystyle{amsplain}
\bibliography{bibliography}

\end{document}